\documentclass[a4paper]{article}
 \usepackage{epsfig,a4}
\usepackage{enumerate}
\usepackage{tikz}
\usepackage{comment}
\usepackage{amsfonts,amsmath,amsthm,amssymb}
\newtheorem{theorem}{Theorem}[section]
\newtheorem*{lemma*}{Lemma}
\newtheorem{corollary}[theorem]{Corollary}

\newtheorem{lemma}[theorem]{Lemma}
\newcommand{\ds}{\displaystyle}

\newcommand{\C}{\mathcal{C}}
\newcommand{\M}{\mathcal{M}}
\newcommand{\K}{\mathcal{K}}
\newcommand{\B}{\mathcal{B}}
\newcommand{\G}{\mathcal{G}}
\newcommand{\ex}{\mathbf{E}\,}
\newcommand{\var}{\mathbf{Var}\,}
\newcommand{\mathbbm}{\mathbf}

\usetikzlibrary{arrows,shapes}
\tikzstyle{vertex}=[circle,fill=black!15,minimum size=12pt,inner sep=0pt,
font=\footnotesize, text=black!50]
\tikzstyle{selected vertex} =
[vertex, fill=black!100,text=white, font=\bf]
\tikzstyle{selected point} =
[vertex, fill=black!100, minimum size = 4pt]
\tikzstyle{point} =
[vertex, fill=black!30, minimum size = 4pt]
\tikzstyle{edge} = [draw,black!30,-]
\tikzstyle{selected edge} = [draw,thick,-]
\tikzstyle{strong edge} = [draw,ultra thick,-]
\tikzstyle{weight} = [font=\small]
\tikzstyle{ignored edge} = [draw,line width=5pt,-,black!20]

\title{Random planar maps and graphs\\ with minimum degree two and three}
\author{Marc Noy \qquad  Lander Ramos
\\ \\ Departament de Matem\`{a}tiques \\ Universitat Polit\`{e}cnica de Catalunya\\
Barcelona Graduate School of Mathematics %
\footnote{Jordi Girona 1--3, 08034 Barcelona, Spain.
E-mail: \texttt{marc.noy@upc.es,landertxu@gmail.com}. Supported by the Spanish Ministerio de Econom\'{i}a y Competitividad grants MTM2014-54745-P and MDM-2014-0445.}}
\date{}

\begin{document}
\maketitle

\begin{abstract}
We find precise  asymptotic estimates for the number  of planar maps and graphs with a condition on the minimum degree, and properties of random graphs from these classes. In particular we show that the size of the largest tree attached to the core of a random planar graph is of order $c \log(n)$ for an explicit constant $c$. These results  provide new information on the structure of random planar graphs.
\end{abstract}

\section{Introduction}

The main goal of this paper is to enumerate planar graphs subject to a condition on the minimum degree $\delta$, and to analyze the corresponding random planar  graphs. Asking for $\delta\ge1$ is not very interesting, since a random planar graph contains in expectation a constant number of isolated vertices. The condition $\delta\ge2$ is directly related to the concept of the core of a graph. Given a connected graph $\G$, its \textit{core} (also called the 2-core in the literature) is the maximum subgraph $\C$ of $\G$ with minimum degree at least two. The core $\C$ is obtained from $\G$ by repeatedly removing vertices of degree one. Conversely, $\G$ is obtained by attaching rooted trees at the vertices of~$\C$. Note that the core of a tree is empty.

The  \emph{kernel} of $\G$ is obtained from $\mathcal{C}$ by 
contracting all the induced paths between vertices of degree greater than 2
to a single edge. The kernel has minimum degree at least three, and $\C$ can be recovered from $\K$ by replacing edges with induced paths. Notice that $\G$ is planar if and only $\C$ is planar, in turn if and only if $\K$ is planar.

As shown in Figure \ref{Fig:graph}, the kernel may have loops and multiple edges, which must be taken into account since our goal is to analyze simple graphs. Another issue is that when replacing loops and multiple edges with paths
the same graph can be produced several times. To this end we  weight multigraphs  according to  the number of loops and edges of each multiplicity. We remark that the concepts of core and kernel of a graph  are instrumental in the classical theory of random graphs \cite{giant,probplanar}.

\begin{figure}[htp]
\begin{center}
\begin{tabular}{|c|c|c|}
\hline
&&\\
$\G$ & $\C$ & $\K$ \\
\begin{tikzpicture}[scale=0.7, auto,swap]
    \foreach \pos/\name in {{(0,-2)/1}, {(0,2)/2}, {(-1,0)/4},
        {(-2,0)/6}, {(-3,0)/3}, {(-1.5,1)/5}, {(-2,-.66)/7},
		{(-1,-1.33)/8}, {(-2,0)/6}, {(1,-1.5)/10}, {(1,-.5)/9},
		{(1,1.5)/13}, {(-3,1)/12}, {(-3,-1)/11}, {(-1.5,2)/14},
		{(1,2.66)/15}, {(-2,-2)/16}}
        \node[selected vertex] (\name) at \pos {$\name$};
	\foreach\source/\dest in {1/2,1/4,2/4,1/8,1/10,2/5,2/13,3/5,3/6,
	3/7,3/11,3/12,4/6,7/8,9/10,5/14,2/15,13/15,1/16,7/16}
	\path[selected edge](\source) -- (\dest);
\end{tikzpicture}
 &
\begin{tikzpicture}[scale=0.7, auto,swap]
    \foreach \pos/\name in {{(0,-2)/1}, {(0,2)/2}, {(-1,0)/4},
                            {(-2,0)/6}, {(-3,0)/3}, {(-1.5,1)/5}, {(-2,-.66)/7},
							 {(-1,-1.33)/8}, {(-2,0)/6},
							 {(1,1.5)/13},  {(1,2.66)/15}, {(-2,-2)/16}}
        \node[selected vertex] (\name) at \pos {$\name$};
    \foreach \pos/\name in {{(1,-1.5)/10}, {(1,-.5)/9},{(-3,1)/12}, {(-3,-1)/11},{(-1.5,2)/14}}
        \node[vertex] (\name) at \pos {$\name$};
	\foreach\source/\dest in
	 {1/2,1/4,2/4,1/8,2/5,3/5,3/6,3/7,4/6,7/8,1/16,7/16,2/13,2/15,13/15}
	\path[selected edge](\source) -- (\dest);
	\foreach\source/\dest in {1/10,3/11,3/12,9/10,5/14}
	\path[edge](\source) -- (\dest);
\end{tikzpicture}
&
\begin{tikzpicture}[scale=0.7, auto,swap]
    \foreach \pos/\name in {{(0,-2)/1}, {(0,2)/2}, {(-1,0)/4},
                            {(-2,0)/6},{(-2,-.66)/7}, {(-3,0)/3}}
        \node[selected vertex] (\name) at \pos {$\name$};
    \foreach \pos/\name in {{(1,-1.5)/10}, {(1,-.5)/9},
							 {(1,1.5)/13}, {(-3,1)/12}, {(-3,-1)/11}, {(-1.5,1)/5},
							 {(-1,-1.33)/8}, {(-2,0)/6},{(-1.5,2)/14}, {(1,2.66)/15}, {(-2,-2)/16}}
        \node[vertex] (\name) at \pos {$\name$};
	\foreach\source/\dest in {1/2,1/4,2/4,2/3,3/4,3/7,1/7}
	\path[strong edge](\source) -- (\dest);
	\foreach\source/\dest in
	 {1/10,2/13,3/11,3/12,9/10,5/14,1/16,7/16,2/15,13/15}
	\path[edge](\source) -- (\dest);
	\path[strong edge] (1) edge [bend left] (7);
	\path[strong edge, every loop/.style={looseness=15}] (2)
	edge [in=-45,out=45,loop] (2);
\end{tikzpicture}
\\ \hline
\end{tabular}

\end{center}
\caption{Core and kernel of a graph.}\label{Fig:graph}
\end{figure}
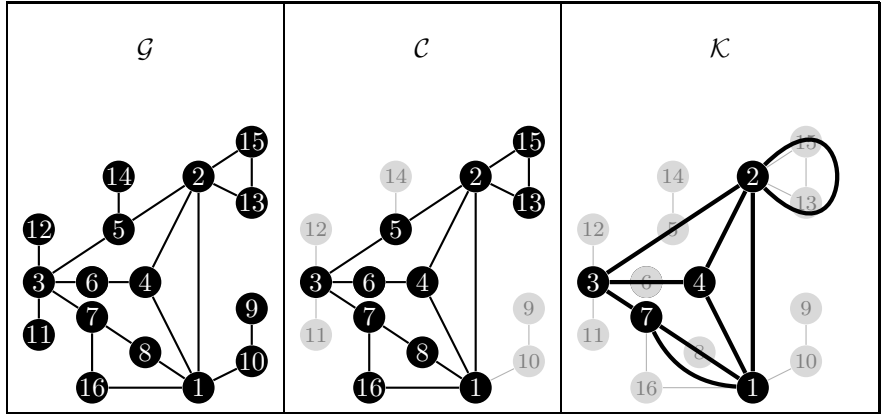

For the sake of brevity, it is convenient to introduce the following definitions: a \emph{2-graph} is a connected graph with minimum degree at least two, and a \emph{3-graph} is a  connected graph with minimum degree at least three.
In order to enumerate planar 2- and 3-graphs we use generating functions.
From now on all graphs are labelled and generating functions are of the exponential type. Let $c_n, h_n$  and $k_n$ be, respectively, the number of  planar connected graphs, 2-graphs and 3-graphs with $n$ vertices, and let
$$
    C(x) = \sum c_n {x^n \over n!}, \qquad
    H(x) = \sum h_n {x^n \over n!}, \qquad
    K(x) = \sum k_n {x^n \over n!}
$$
be the associated generating functions. Also, let $t_n=n^{n-1}$ be the number of (labelled) rooted trees with $n$ vertices and let $T(x)=\sum t_n x^n/n!$. The decomposition of a connected graph into its core and the attached trees implies the following equation
\begin{equation}\label{cores}
    C(x) = H(T(x)) + U(x),
\end{equation}
where $U(x)= T(x)- T(x)^2/2$ is the generating functions of \emph{unrooted} trees.
Since $T(x)=x e^{T(x)}$, we can invert the above relation and obtain
$$
    H(x) = C(xe^{-x})  -x + {x^2\over 2}.
    $$
The equation defining $K(x)$ is more involved and requires the bivariate generating function
$$
C(x,y) = \sum c_{n,k}\,y^k {x^n \over n!},
$$
where $c_{n,k}$ is the number of connected planar graphs with $n$ vertices and $k$ edges. We can express $K(x)$ in terms of $C(x,y)$ as
\begin{equation}\label{kernels}
    K(x) = C(A(x),B(x)) + E(x),
\end{equation}
where $A(x),B(x),E(x)$ are explicit elementary functions (see Section \ref{sec:equations}).

From the expression  of $C(x)$ as the solution of a system of functional-differential equations \cite{gn}, it was shown that
$$
    c_n \sim \kappa n^{-7/2} \gamma^n n!,
$$
where $\kappa\approx 0.4104\cdot 10^{-5}$ and $\gamma \approx 27.2269$ are computable constants. In addition, analyzing  the bivariate generating function $C(x,y)$ it is possible to obtain results on the number of edges and other basic parameters in random planar graphs. Our main goal is to  extend these results to planar 2-graphs and 3-graphs.

Using Equations (\ref{cores}) and (\ref{kernels}) we obtain precise asymptotic estimates for the number of planar 2- and 3-graphs:
$$
\renewcommand{\arraystretch}{1.5}
\begin{array}{llll}
    h_n \sim & \kappa_2 n^{-7/2} \gamma_2^n n!,
    \qquad &\gamma_2 \approx 26.2076, & \kappa_2 \approx 0.3724\cdot 10^{-5}, \\
    k_n \sim & \kappa_3 n^{-7/2} \gamma_3^n n!,
    \qquad &\gamma_2 \approx 21.3102, & \kappa_3 \approx 0.3107\cdot 10^{-5}.
\end{array}
$$
As is natural to expect, $h_n$ and $k_n$ are exponentially smaller than $c_n$.
Also, the number of 2-connected planar graphs is known to be asymptotically
$\kappa_{c} n^{-7/2} 26.1841^n n!$ (see \cite{bgw}), smaller than the number of 2-graphs.
This is consistent, since a 2-connected graph has minimum degree at least two, but not conversely. 

By enriching Equations (\ref{cores}) and (\ref{kernels}) taking into account the number of edges, we prove that the number of edges in random planar 2-graphs and 3-graphs are both asymptotically normal with linear expectation and variance.
The expected number of edges in connected planar graphs was shown to be \cite{gn} asymptotically $\mu n$, where $\mu \approx 2.2133$. We show that the corresponding constants for planar 2-graphs and 3-graphs are
$$
\mu_2\approx 2.2614, \qquad \mu_3\approx 2.4065.
$$
This conforms to our intuition that increasing the minimum degree should increase the expected number of edges.

We also analyze the size $X_n$ of the core in a random connected planar graph, and the size $Y_n$ of the kernel in a random planar 2-graph. We show that both variables are asymptotically normal with linear expectation and variance and that
$$
\renewcommand{\arraystretch}{1.3}
\begin{array}{lll}
    \ex X_n \sim &\lambda_2 n,  \qquad &\lambda_2 \approx 0.9618, \\
    \ex Y_n \sim &\lambda_3 n, \qquad &\lambda_3 \approx 0.8259.
\end{array}
$$
The value of $\lambda_2$ has been recently found by McDiarmid \cite{colin} using alternative methods.
Also, we remark that the expected size of the largest block (2-connected component) in random connected planar graphs is asymptotically $0.9598n$ \cite{3-conn}. Again this is consistent since the largest block is contained in the core but not conversely. 

The picture is completed by analyzing the size of the trees attached to the core. We show that for fixed $k \ge 1$  the number of trees with $k$ vertices attached to the core is asymptotically normal with linear expectation and variance. The expected value is asymptotically
$$
C {k^{k-1} \over k! } \rho^k n,
$$
where $C>0$ is a constant and $\rho = \gamma^{-1} \approx 0.03673$ is the radius of convergence of $C(x)$.
For $k$ large, the previous quantity grows like
$$
{C \over \sqrt{2\pi}}\cdot k^{-3/2} (\rho e)^k n.
$$
This quantity is negligible when $k \gg \log(n)/(\log(1/\rho e))$. Using the method of moments, we show that the size $L_n$ of the largest tree attached to the core satisfies
$$
\frac{L_n}{\log n} \to     {1 \over \log(1/\rho e)}  \approx 0.4340 \qquad \hbox{in probability}.
$$
This result provides new  information on the structure of random planar graphs.

Our last result concerns the distribution of the vertex degrees in random planar 2-graphs and 3-graphs. We show that for each fixed $k\ge2$ the probability that a random vertex has degree $k$ in a random planar 2-graph tends to a positive constant $d_H(k)$, and for each fixed $k\ge3$ the probability that a random vertex has degree $k$ in a random planar 3-graph tends to a positive constant $d_K(k)$. Moreover
$  \sum _{k\ge2} p_H(k) =     \sum _{k\ge3} p_K(k) = 1$,
and the probability generating functions
$$
 p_H(w) =\sum _{k\ge2} d_H(k)w^k, \qquad     p_K(w)= \sum _{k\ge3} d_K(k)w^k
 $$
are computable in terms of the analogous probability generating function $p_C(w)$ of connected planar graphs, which was fully determined in \cite{degree}.

The previous results show that almost all planar 2-graphs have a vertex of degree two, and almost all planar 3-graphs have a vertex of degree three. Hence asymptotically all our results hold also for planar graphs with minimum degree exactly two and three, respectively.  For the sake of conciseness, we will not repeat for each of our results the corresponding statement for arbitrary graphs of minimum degree exactly two or three.
In addition, all our  results for connected planar graphs  easily extend to arbitrary planar graphs. This is because the expected size of the largest component in a random planar graph is $n-O(1)$ (see \cite{3-conn}). For simplicity, we state our results only for arbitrary planar graphs.

It is natural to ask why we stop at minimum degree three. The reason is that
there seems to be no combinatorial decomposition allowing to deal with planar graphs of minimum degree four or five (a planar graph has always a vertex of degree at most five). It is already an open problem to enumerate 4-regular planar graphs. In contrast, the enumeration of cubic planar graphs was  solved in \cite{cubic}.

The contents of the paper are as follows. In Section \ref{sec:prelim} we review some technical preliminaries needed in the paper.
In Section \ref{sec:maps} we find similar results for planar maps, that is, connected planar graphs with a fixed embedding. They are simpler to derive and serve as a preparation for  the results on planar graphs, while at the same time they appear to be  new and interesting by themselves.
In Section \ref{sec:equations} we find equations linking the generating functions of connected graphs, 2-graphs and 3-graphs; to this end we must consider multigraphs as well as simple graphs. In Section~\ref{sec:graphs} we use singularity analysis in order to prove our main results on asymptotic enumeration and properties of random planar 2-graphs and 3-graphs.
The analysis of the distribution of the degree of the root, which is technically  more involved, is deferred to Section~\ref{sec:degree}. We conclude with some remarks and open problems.

\section{Preliminaries}\label{sec:prelim}

We assume familiarity with the basic results of analytic combinatorics as described in \cite{flajolet}. 
Given a complex number $\rho \ne0$, a $\Delta$-domain at $\rho$ is an open set of the form 
$$
	\Delta(R,\phi) = \{  z \colon |z|<R, z \ne \rho, |\arg(z-\rho)|> \phi \}.
$$
A singularity of $f(z)$ is a point where $f(z)$ ceases to be analytic. 
A dominant singularity is one of minimum modulus. We say that $f(z)$ is $\Delta$-analytic at $\rho$ if it is anlytic in a $\Delta$-domain at $\rho$. 
We will need the following result  \cite[Corollary VI.1]{flajolet}. 

\begin{theorem}[Transfer Theorem]
	If $f(z)$  has a unique dominant singularity at 
	$\rho$ at which is $\Delta$-analytic and satisfies the estimate
$$
    f(z) \sim (1-z/\rho)^{-\alpha}, \qquad z \to \rho,
    $$
with $\alpha \not\in \{0,-1,-2,\dots \}$, then the coefficients of $f(z)$ satisfy
$$
    [z^n]f(z) \sim {n^{\alpha-1} \over \Gamma(\alpha)} \,\rho^{-n}.
$$
\end{theorem}

We also need  and a simplified version of \cite[Theorem IX.8]{flajolet}.
\begin{theorem}[Quasi-powers Theorem]
%
Let the $X_n$ be non-negative discrete random variables with probability generating functions $p_n(u)$. Assume that, uniformly in a fixed complex neighbourhood of $u = 1$ 
$$
 p_n(u)=A(u)·B(u)^n \left(1+O\left(\frac{1}{n}\right)\right),
$$
where $A(u), B(u)$ are analytic at $u = 1$ and $A(1) = B(1) = 1$. Assume that $B(u)$ satisfies $B''(1) + B'(1) - B'(1)^2 \ne 0$. 

Then the distribution of $X_n$ is, after standardization, asymptotically Gaussian, 
and the mean and variance satisfy 
$$
    \ex X_n \sim \left(\frac{B'(1)}{B(1)}\right) n,
    \qquad 
    \var X_n \sim \left({B''(1) \over B(1)} + {B'(1) \over B(1)}  - \left({B'(1) \over B(1)}\right)^2 \right) n.
$$
\end{theorem}In our applications we will have $B(u)=\rho(1)/\rho(u)$, where $\rho(u)$ 
will be the dominant singularity (as a function of $z$) of a bivariate generating function $f(z,u)$. The former expressions become then 
$$
    \ex X_n \sim \left({-\rho'(1) \over \rho(1)} \right) n,
    \qquad \var X_n \sim \left(-{\rho''(1) \over \rho(1)} - {\rho'(1) \over \rho(1)}      + \left({\rho'(1) \over \rho(1)}\right)^2 \right) n.
$$

In order to apply the former results we need to show that the corresponding generating functions are $\Delta$-analytic at suitable singularities. This is relatively simple for planar maps, since we have explicit algebraic expressions for the generating functions, but it is  rather more involved for planar graphs. The expressions obtained in Section \ref{sec:equations} are not enough for this purpose and we have to use alternative equations related to the decomposition of connected graphs into 2-connected components (see~Section \ref{sec:graphs}). Some of these derivations are rather long and are given in the Appendix.
 Several of the arguments we use may have applications in related situations where  $\Delta$-analyticity has to be guaranteed. 

In Section \ref{sec:graphs}  we need the following result from \cite{maxdegsp}. It deals with the maximum degree of random graphs and can be adapted to other extremal parameters
such as  the size of the largest tree attached to the core.
In can be thought of as a kind of `master theorem' for analyzing the maximum degree and related extremal parameters. 

\begin{theorem}\label{th:master}
Let $d_{n,k}$ denote the probability that a randomly selected vertex of a certain
class of random graphs of size $n$ has degree $k$, and let $d_{n,k,\ell}$ denote the probability that two
different randomly selected (ordered) vertices have degrees $k$ and $\ell$. Suppose that we have
the following properties.

\begin{enumerate}
\item
There exists a limiting degree distribution $\overline{d}_k$ $(k > 1)$ with an asymptotic behaviour of
the form
$$\log \overline{d}_k\sim k \log q \quad (k \to \infty),
$$
where $q$ is a real constant with $0 <q< 1$.

\item We have, as $n \to \infty, k \to \infty, \ell \to \infty$, and uniformly for $k,\ell \le    C \log n$ (for an arbitrary
constant $C > 0$)
$$
d_{n,k} \sim \overline{d}_k \quad \hbox{and}
 d_{n,k,\ell} \sim \overline{d}_k \overline{d}_\ell.
 $$
 
 \item 
There exists $\overline{q} < 1$ such that, uniformly for all $n, k, \ell \ge  1$,
$$
d_{n,k} = O(\overline{q}^k) \quad \hbox{and} \quad
d_{n,k} = O(\overline{q}^{k+\ell}).
$$
\end{enumerate}
Let $\Delta_n$ denote the maximum degree of a random graph of size $n$ in this class. Then
$$\frac{\Delta_n}{\log n} \to  \frac{1}{\log(1/q)} \quad \hbox{in probability},
$$
and
$$
\ex \Delta_n \sim \frac{1}{\log(1/q)} \log n \quad (n \to \infty). 
$$
\end{theorem}

\section{Planar maps}\label{sec:maps}

We recall that a planar map is a connected planar multigraph embedded in the plane up to homeomorphism. A map is rooted if one of its edges is distinguished and oriented. In this way a rooted map has a root edge and a root vertex (the tail of the root edge). We define the root face as the face to the right of the  root edge. A rooted map has no automorphisms, in the sense that every vertex, edge and face is distinguishable. From now on all maps are planar and rooted. We stress the fact  that maps may have loops and multiple edges.

The enumeration of rooted planar maps was started by Tutte in his seminal paper
 \cite{tutte}. Let $m_n$ be the number of rooted maps with $n$ edges, with the
  convention that $m_0=0$.
   Then
$$
    m_n = {2 \cdot 3^n \over (n+2)(n+1)} { 2n \choose n}, \quad n\geq 1
$$
The generating function $M(z) = \sum_{n\ge0} m_n z^n$ is equal to
\begin{equation}\label{th:Mn}
   M(z) = {18z-1 + (1-12z)^{3/2} \over 54z^2}-1.
\end{equation}
Either from the explicit formula or from the expression for $M(z)$ and the transfer theorem, it follows that
\begin{equation}\label{eq:estimatesmaps}
    m_n \sim {2 \over \sqrt\pi} \, n^{-5/2} 12^n.
\end{equation}
If $m_{n,k}$ is the number of maps with $n$ edges and degree of the root face equal to $k$, then $M(z,u) = \sum m_{n,k}  z^n u^k$ satisfies the equation
\begin{equation}\label{eq:maps}
    M(z,u) = zu^2(M(z,u)+1)^2 + uz \left({uM(z,u)-M(z,1)\over u-1}+1\right).
\end{equation}
By duality, $M(z,u)$ is also the generating function of maps
in which $u$ marks the degree of the root vertex. The empty map is not  included so that $m_0=0$. 

The core $\C$ of a map $\M$ is  obtained, as for graphs,
by removing repeatedly vertices of degree one,
so that $\C$ has minimum degree at least two (the core is empty if and only if $\M$ is a tree). Then 
 $\M$ is obtained from $\C$ by placing a planar tree
at each corner (pair of consecutive half-edges) of $\C$. This is equivalent
to replacing each edge with a non-empty planar tree rooted at an edge.
The number $t_n$ of planar trees with $n\geq 1$ edges
is equal to the $n$-th Catalan number
and the generating function $T(z) = \sum t_n z^n$ satisfies
$$
 T(z) = {1 \over 1-z(1+T(z))}-1.
    $$
We define a 2-map as a map with minimum degree at least two,
and a 3-map as a map with minimum degree at least three.
Let $h_n$ and $k_n$ be, respectively,
the number of 2-maps and 3-maps with $n$ edges.

\begin{theorem}\label{th:maps}
The  generating functions $H(z)$ and $K(z)$ of 2-maps and 3-maps, respectively,   are given by
$$
\begin{array}{ll}
  H(x) = \ds {1-x \over 1+x} \left(M\left({x \over (1+x)^2} \right) -x\right) =    x+3x^2+16x^3+96x^4+624x^5 + \cdots, \\
K(x) = \ds{H\ds\left({x\over 1+x}\right) -x \over 1+x} =
    2 z^2+9 z^3+47 z^4+278 z^5+ \cdots
\end{array}
$$
The following estimates hold:
\begin{equation}\label{eq:estimates23maps}
  h_n  \sim \kappa_2 n^{-5/2} (5+2\sqrt6)^n, \qquad
  k_n  \sim \kappa_3 n^{-5/2} (4+2\sqrt6)^n,
\end{equation}
where
$$\kappa_2 = \frac{2}{\sqrt{\pi}} \left(\frac{2}{3}\right)^{5/4}
\approx 0.6797,\quad
\kappa_3 = \frac{2}{\sqrt{\pi}}\left(4-4\sqrt{\frac{2}{3}}\right)^{5/2}
\approx 0.5209.
$$
\end{theorem}

\begin{proof}
The decomposition of a map into its core and the collection of trees attached to the corners implies the following equation:
\begin{equation}\label{eq:MH}
    M(z) = T(z)+ H\left(T(z)\right){1+T(z)\over 1-T(z)}   .
\end{equation}
The first summand corresponds to the case where the map is a tree, and the second one where the core is non-empty:
each edge is replaced with a non-empty tree whose root corresponds
to the original edge. The factor
$$
{1+T(z) \over  1-T(z)} = 1+{2\,T(z) \over 1-T(z)}
$$
is interpreted as follows.
The first summand  corresponds to the case where
the root of the map belongs to the core, and the second one to the
case where it is in a pendant rooted tree $\tau$, which we place at the
left-back corner of the root edge of the core. In this case there
is a non-empty sequence of non-empty trees from the root
edge~$e$ of $\tau$ to the root edge of the core, and  the factor $2$
distinguishes the two possible orientations of $e$.

In order to invert the former relation let $x=T(z)$, so that
$$
z={x\over(1+x)^2}.
$$
We obtain
\begin{equation}\label{eq:HM}
    H(x) = {1-x \over 1+x} \left(M\left({x \over (1+x)^2} \right) -x\right).
\end{equation}

\begin{figure}[thb]
\begin{center}
\begin{tabular}{|c|c|c|}
\hline
&&\\
$\M$ & $\C$ & $\K$ \\
\begin{tikzpicture}[scale=1.1, auto,swap]
    \foreach \pos/\name in {{(0,0)/1}, {(2,0)/2}, {(2,3)/3},
        {(1,3)/4}, {(1,3.5)/5}, {(0,3)/6}, {(0,3.5)/7},
		{(-0.5,3)/8}, {(-0.5,2)/9}, {(0,2)/10}, {(-0.5,1)/11},
		{(0,1)/12}, {(-0.5,0.5)/13}, {(0.5,2.25)/14}, {(1,2.5)/15},
		{(1,1)/16}, {(1,1.5)/17}, {(1.5,1.5)/18}, {(1.33,1)/19},
		{(1.66,1)/20}}
        \node[selected point] (\name) at \pos{};
	\foreach\source/\dest in {1/2,1/12,1/17,2/3,3/4,3/17,4/6,6/7,6/8,7/8,
	6/14,6/10,9/10,10/12,11/12,12/13,14/15,14/17,16/17,17/18,18/19,18/20}
	\path[selected edge](\source) -- (\dest);
	\begin{scope}[very thick, ->]
	  \draw (0.99,0)--(1,0);
	\end{scope}
	\path[selected edge](2) edge [bend right] (3);
	\path[selected edge](4) edge [bend right] (5);
	\path[selected edge](4) edge [bend left] (5);
	\path[selected edge, every loop/.style={looseness=30}] (3)
	edge [in=0,out=90,loop] (3);
\end{tikzpicture}
 &
\begin{tikzpicture}[scale=1.1, auto,swap]
    \foreach \pos/\name in {{(0,0)/1}, {(2,0)/2}, {(2,3)/3},
        {(1,3)/4}, {(1,3.5)/5}, {(0,3)/6}, {(0,3.5)/7},
		{(-0.5,3)/8}, {(0,2)/10},
		{(0,1)/12}, {(0.5,2.25)/14},
		{(1,1.5)/17}}
        \node[selected point] (\name) at \pos{};
    \foreach \pos/\name in {{(-0.5,2)/9}, {(-0.5,1)/11},
		{(-0.5,0.5)/13}, {(1,2.5)/15},
		{(1,1)/16}, {(1.5,1.5)/18}, {(1.33,1)/19},
		{(1.66,1)/20}}
        \node[point] (\name) at \pos{};
	\foreach\source/\dest in {1/2,1/12,1/17,2/3,3/4,3/17,4/6,6/7,6/8,7/8,
	6/14,6/10,10/12,14/17}
	\path[selected edge](\source) -- (\dest);
	\foreach\source/\dest in {9/10,11/12,12/13,14/15,16/17,17/18,18/19,18/20}
	\path[edge](\source) -- (\dest);
	\begin{scope}[very thick, ->]
	  \draw (0.99,0)--(1,0);
	\end{scope}
	\path[selected edge](2) edge [bend right] (3);
	\path[selected edge](4) edge [bend right] (5);
	\path[selected edge](4) edge [bend left] (5);
	\path[selected edge, every loop/.style={looseness=30}] (3)
	edge [in=0,out=90,loop] (3);
\end{tikzpicture}
&
\begin{tikzpicture}[scale=1.1, auto,swap]
    \foreach \pos/\name in {{(0,0)/1}, {(2,0)/2}, {(2,3)/3},
        {(1,3)/4}, {(0,3)/6},
		{(1,1.5)/17}}
        \node[selected point] (\name) at \pos{};
    \foreach \pos/\name in {{(1,3.5)/5}, {(0,3.5)/7},
		{(-0.5,3)/8}, {(0,2)/10},
		{(0,1)/12}, {(0.5,2.25)/14}, {(-0.5,2)/9}, {(-0.5,1)/11},
		{(-0.5,0.5)/13}, {(1,2.5)/15},
		{(1,1)/16}, {(1.5,1.5)/18}, {(1.33,1)/19},
		{(1.66,1)/20}}
        \node[point] (\name) at \pos{};
	\foreach\source/\dest in {1/2,1/6,1/17,2/3,3/4,3/17,4/6,
	6/17}
	\path[strong edge](\source) -- (\dest);
	\foreach\source/\dest in {9/10,11/12,12/13,14/15,16/17,17/18,18/19,18/20,
	6/7,7/8,6/8}
	\path[edge](\source) -- (\dest);
	\begin{scope}[very thick, ->]
	  \draw (0.99,0)--(1,0);
	\end{scope}
	\path[strong edge](2) edge [bend right] (3);
	\path[edge](4) edge [bend right] (5);
	\path[edge](4) edge [bend left] (5);
	\path[strong edge, every loop/.style={looseness=30}]
	(3)edge [in=0,out=90,loop] (3);
	\path[strong edge, every loop/.style={looseness=30}] (6)
	edge [in=180,out=90,loop] (6);
	\path[strong edge, every loop/.style={looseness=30}] (4)
	edge [in=45,out=135,loop] (4);
\end{tikzpicture}
\\ \hline
\end{tabular}

\end{center}
\caption{Core and kernel of a map.}
\label{fig:map}
\end{figure}
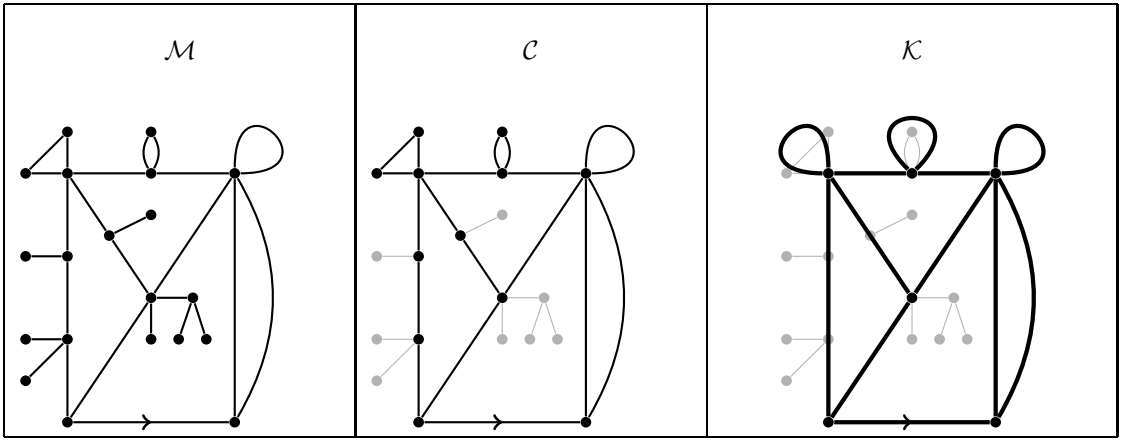

Let now $\C$ be a 2-map. The kernel $\K$ of $\C$ is defined as follows: replace every maximal path of vertices of degree
two in $\C$ with a single edge (see Figure \ref{fig:map}). Clearly $\K$ is a 3-map and $\C$ can be obtained by replacing edges in $\K$ with paths.
It follows that
\begin{equation}\label{eq:HK}
    H(z) = K\left({z \over 1-z} \right){1 \over 1-z}+ {z \over 1-z}.
\end{equation}
The first term corresponds to the substitution of paths for edges, and the extra factor $1/(1-z)$  indicates where to locate the new root edge
in the path replacing the original root edge. The last term corresponds to cycles, whose kernel is empty. Inverting the relation $x=z/(1-z)$ we obtain
\begin{equation}\label{eq:KH}
    K(x) = \ds{H\ds\left({x\over 1+x}\right) -x \over 1+x}.
\end{equation}

In order to obtain asymptotic estimates for $h_n$ and $k_n$ we need to locate the dominant singularities of $H(z)$ and $K(z)$ and show that these functions are analytic on suitable $\Delta$-domains.
$M(z)$ has a unique singularity  at  $\rho = 1/12$ and is analytic in $\mathbb{C}$ minus the ray $[1/12,+\infty)$, and $T(z)$ is singular only at $1/4$. Hence $H(z)$ has a singularity at $\sigma = \rho T(\rho)^2 = 5 - 2 \sqrt6$. We show next that $H(z)$ is analytic in $|x| < \sigma$ and has no other singularities in $|x|=\sigma$. By continuity it is $\Delta$-analytic at $\sigma$. 

From Equation (\ref{eq:HM}), the singularities of $H$ are at  $-1$ and at the  points $x$
where $t={x/(1+x)^2}\in [1/12,\infty)$. 
We show that these points either satisfy $|x|\ge1$ or 
belong to the real segment $[\sigma,1)$. If we solve the equation for $x$ we get
$$
x={1-2t\pm\sqrt{1-4t}\over 2t},\qquad t\in [1/12,\infty).
$$
We analyze two cases. For $t > 1/4$ we can rewrite $x = (1-2t	\pm i\sqrt{4t-1})/ 2t$ and obtain $|x|=1$. 
When $1/12 \le t \le 1/4$, $x$ must be real.
Consider the solution $x(t)= (1-2t+\sqrt{1-4t})/ 2t$. 
It is non-increasing since the derivative 
$$
x'(t) = -{1-2t+\sqrt{1-4t}
\over 2\sqrt{1-4t}\cdot t^2}
$$
is negative. Since $x(1/4)=1$ it follows that $x\geq 1$.
For the solution $x(t)= (1-2t-\sqrt{1-4t})/ 2t$ the derivative
is positive and $x(1/12) = \sigma$. Hence $x\ge\sigma$.

From Equation (\ref{eq:KH}) it follows that $K(x)$ has a singularity at 
 $\tau = \sigma/(1-\sigma) = (\sqrt6-2)/4$.  A similar argument as before shows that $K(x)$ is   $\Delta$-analytic at $\tau$.

The singular expansion of $M(z)$ at the singularity $z=1/12$ can be obtained directly from the explicit formula (\ref{th:Mn}), and is equal to
$$
    M(z) = {1 \over 3} -{4\over 3}Z^2 + {8\over 3}Z^3 + O(Z^4),
$$
where $Z =\sqrt{1-12z}$.
Plugging this expression into (\ref{eq:HM}) and expanding gives
$$
    H(x) = H_0 +H_2 X^2 + \frac{8}{3}\left(\frac{2}{3}\right)^{5/4}
     X^3 + O(X^4),
$$
where now $X=\sqrt{1-x/\sigma}$.
A similar computation using  (\ref{eq:KH}) gives
$$
    K(x) = K_0 +K_2 X^2 +
    \frac{8}{3}\left(4-4\sqrt{\frac{2}{3}}\right)^{5/2} X^3 + O(X^4),
$$
where  $X=\sqrt{1-x/\tau}$.

The estimates for $h_n$ and $k_n$ follow by the transfer theorem and  the value $\Gamma(-3/2)=4\sqrt\pi/3$.
\end{proof}

\noindent
For future reference we display  the dominant singularities for 2- and 3-maps, respectively:
$$
    \sigma = 5-2\sqrt6, \qquad\tau = {\sqrt6-2 \over 4}.
    $$

Our next result is a limit law for the size of the core and the kernel in random maps.

\begin{theorem}
The size $X_n$ of the core of a random map with $n$ edges, and the size $Y_n$ of the kernel of  a random 2-map with  $n$ edges
are asymptotically Gaussian with
$$
\renewcommand{\arraystretch}{2}
\begin{array}{ll}
    \ex  X_n \sim \ds{\sqrt 6 \over 3}n \approx 0.8165n, &  \quad   \var X_n \sim \ds\frac{n}{6} \approx 0.1667n,\\
    \ex Y_n \sim (2\sqrt6 -4)n \approx 0.8990n, &  \quad \var Y_n
    \sim (18\sqrt{6}-44)n\approx 0.0908n.\\
\end{array}
        $$
The size $Z_n$ of the kernel of a random map with $n$ edges is also
asymptotically Gaussian with
$$
    \ex Z_n \sim
    \left(4-{4\sqrt6\over3}\right)n \approx 0.7340n,\qquad
    \var Z_n \sim \left({128\over 3} - {52\over 3}\sqrt{6}\right)n
    \approx 0.2088n.
    $$
\end{theorem}

\begin{proof}
If $u$ marks the size of the core in maps, then an immediate extension of (\ref{eq:MH}) yields
\begin{equation}\label{eq:MHu}
    M(z,u) = H\left(uT(z)\right){1+T(z)\over 1-T(z)}   +T(z).
\end{equation}
It follows that a dominant singularity $\xi(u)$ of the univariate function $z \mapsto M(z,u)$ is given by
$uT(\xi(u))=\sigma$. Inverting this relation we obtain 
$$
    \xi(u) = {\sigma u\over (\sigma+u)^2}.
$$
Consider $u \in \mathbb{C}$ close to 1. If $|z| = \xi(u)$
but $z \neq \xi(u)$ then $M(z,u)$ is analytic at $z$. Indeed,
for such a $|z|$ we have:
$$
|uT(z)| = |u||T(z)| < |u|T(|z|) = |u|{\sigma \over |u|} = \sigma.
$$
Now we can apply the quasi-powers theorem, so that the distribution is asymptotically Gaussian with linear expectation and variance.
An easy calculation gives
$$
         - {\xi'(1) \over \xi(1)} = {\sqrt6 \over 3}, \qquad
     -{\xi''(1) \over \xi(1)} - {\xi'(1) \over \xi(1)}
     + \left({\xi'(1) \over \xi(1)}\right)^2 = {1\over 6}.
$$

If now $u$ marks the size of the kernel in 2-maps then an extension of (\ref{eq:HK}) gives
\begin{equation}\label{eq:HKu}
    H(z,u) = K\left({uz \over 1-z} \right){1 \over 1-z}+ {z \over 1-z}.
\end{equation}
A dominant  singularity $\chi(u)$ of  $z \mapsto K(z,u)$ is now given by
$$
    \chi(u) = {\tau \over \tau+|u|}.
    $$   
Again, for $u_0$ close enough to 1 the generating function $H(z,u_0)$ can be extended
to a $\Delta$-domain with inner radius $\chi(u_0)$. As before, for
$z$ with $|z|=\chi(u_0)$ but $z \neq \chi(u_0)$ we have:
$$
\left| u_0 {z\over 1-z}\right| = |u_0| \left|{z\over 1-z}\right|
< |u_0| {|z|\over 1-|z|} = |u_0|{\tau\over |u_0|} = \tau.
$$
Therefore the quasi-powers theorem applies and we have
$$
     - {\chi'(1) \over \chi(1)} = {2\sqrt6 -4}, \qquad
     -{\chi''(1) \over \chi(1)} - {\chi'(1) \over \chi(1)}
     + \left({\chi'(1) \over \chi(1)}\right)^2 = 18\sqrt{6}-44.
$$

The last statement concerning $Z_n$ follows by combining equations (\ref{eq:MHu}) and (\ref{eq:HKu}), obtaining an expression of $M(z,u)$ in terms
of $K(z)$, and repeating the same computations as before for the
corresponding singularity function.
\end{proof}

It is interesting to compare the previous result with the known results on the
largest block (2-connected component) \cite{airy}.
The expected size of the largest block in random maps is asymptotically $n/3$,
quite  smaller than the size of the core.
In other words, the core $\C$ consists of the largest  block $\B$
together with smaller blocks attached to $\B$
comprising in total  ${\sqrt6-1 \over 3}n \approx 0.4832n$ edges.
An explanation for this fact is the presence of a linear number of loops,
which belong to the core, but do not belong to the largest block.

\paragraph{Degree distribution.}
Our last result in this section deals with the distribution of the degree of the root vertex in 2-maps and 3-maps.
We let $M(z,u)$ be the GF of maps, where $z$ marks edges and $u$ marks the degree of the root vertex. Similarly, $H(z,u)$ is the GF for 2-maps, and  $T(z,u) = 1/(1-uz(T(z)+1))-1$ for trees, where again $u$ marks the degree of the root. Then we have
$$
    M(z,u) = H\left(T(z), {u (T(z,u)+1) \over T(z)+1}\right) (T(z,u)+1)
    +H(T(z)){T(z,u)\over 1-T(z)}+ T(z,u).
$$
The first term corresponds to the case where the root belongs to the core:
we replace each edge with a tree, and each edge incident
to the root vertex is replaced with a possibly empty tree, where
$u$ marks the degree of the root. The term $T(z)+1$ in the denominator
ensures that an edge is not replaced twice with a tree. The factor
$T(z,u)+1$ allows to place a possibly empty tree in the root corner.
The second term corresponds to the case where the root belongs to
a tree attached to the core: the denominator $1-T(z)$
encodes  a sequence of trees going from the core to the root edge.
The last term corresponds to the case where the core is empty, and therefore
the map is a tree.

If we change variables $x=T(z)$ and $w = u(T(u,z)+1)/(T(z)+1)$, the inverse is
$$
    z ={x \over (1+x)^2}, \qquad u={w(1+x) \over 1+wx}.
$$
The former equation becomes
\begin{equation}\label{eq:2mapsdeg}
    H(x,w) = \ds{M\left(\ds{x \over (1+x)^2}, \ds{w(1+x) \over 1+wx}\right)
    \over 1+wx} -
    {wx \over 1+x} M\left({x\over (1+x)^2}\right) +{1\over 1+wx}+{wx^2\over 1-x} - 1.
\end{equation}
The first terms are
$$
H(x,u)=
{  w}^{2}x+ \left( {w}^{2}+2{w}^{4} \right) {x}^{2}+ \left(
3{w}^{2}+4{w}^{3}+ 4{w}^{4}+5{w}^{6} \right) {x}^{3}+
\cdots  $$
The relationship between $H(z,u)$ and $K(z,u)$ is simpler:
$$
    H(z,u) = K\left({z \over 1-z},u\right) +
    K\left({z \over 1-z}\right){zu^2 \over 1-z} + {zu^2 \over 1-z}.
$$
Inverting gives
\begin{equation}\label{eq:3mapsdeg}
K(x,u) = H\left({x \over 1+x},u\right)
-{xu^2\over 1+x}H\left({x\over 1+x}\right) - {xu^2\over 1+x},
\end{equation}
and the first terms are
$$
K(z,u)= 2u^4z^2+(4u^3+5u^6)z^3+(9u^3+9u^4+15u^5+14u^8)z^4 +  \cdots
$$

In order to analyze $H(z,u)$ and $K(z,u)$ we need the expansion of $M(z,u)$ near the singularity $\rho=1/12.$ Notice that the singularity does not depend on $u$ for $u\sim1$, hence the anliticity in a $\Delta$-domain is granted. 
As we have seen, the expansion of $M(z)$ near $z=1/12$ is
$$
      M(z)=  {1 \over 3} -{4 \over 3}Z^2 + {8 \over 3}Z^3 + O(Z^4),
$$where $Z=\sqrt{1-12z}$. Since $M(z,u)$ satisfies (\ref{eq:maps}) we obtain
\begin{equation}\label{eq:singMapsdeg}
M(z,u) = M_0(u) + M_2(u)Z^2 + M_3(u)Z^3 + O(Z^4).
\end{equation}
A simple computation by indeterminate coefficients gives
$$
    M_3(u) = {8u \over \sqrt{3(2+u)(6-5u)^3}}.
$$
The limiting probability that a random map has a root vertex (or face) of
degree $k$ is equal to
$$
    p_M(k)=  {[u^k][z^n] M(z,u) \over [z^n]M(z)}.
$$
Both coefficients can be estimated using transfer theorems and we get that the probability
generating function of the distribution is given by

\begin{equation}\label{eq:distMaps}
p_M(u) = \sum p_M(k) u^k = {M_3(u) \over M_3(1)} = {\frac {u\sqrt
{3}}{\sqrt { \left( 2+u \right)  \left( 6-5\,u \right) ^{3}}}}.
\end{equation}

Our goal is to obtain analogous results for 2-maps and 3-maps.

\begin{theorem}\label{degree-maps}
Let $p_M(u)$ be as before, and let $p_H(u)$ and $p_K(u)$ be the probability generating functions for the distribution of the root degree in 2-maps and 3-maps, respectively.
Then we have
$$
p_H(u) =  {p_M\left( \ds{u(1+\sigma)\over 1+u\sigma}\right)
\ds{1+\sigma \over 1+u\sigma} -u\sigma \over 1- \sigma},
$$
$$
p_K(u) = {p_H(u)-u^2\sigma\over 1-\sigma},
$$
where $\sigma = 5-2\sqrt{6}$, as in Theorem~\ref{th:maps}.
Furthermore, the limiting probabilities that the degree of the root vertex is equal to $k$ exist, both for 2-maps and 3-maps, and are asymptotically
$$
\renewcommand*{\arraystretch}{1.5}
\begin{array}{l}
 p_H(k) \sim \nu_2 k^{1/2} w_H^k, \qquad  \\
    p_M(k) \sim \nu_3 k^{1/2} w_K^k,
 \end{array}
$$
where $w_H = w_K = \sqrt{2/3} \approx 0.8165$,
$\nu_2 = \sqrt{3(1-\sigma)/ (64\pi)}
\approx 0.1158$,
$\nu_3 =  \sqrt{3/ (64\pi(1-\sigma))} \approx 0.1288$.
\end{theorem}

The correction terms $u\sigma$ in $p_H(u)$ and $u^2\sigma$ in $p_K(u)$
are due to the fact, respectively, that 2-maps have no vertices of degree one and 3-maps no vertices of degree two.

\begin{proof}
Since $M(z,u)$ satisfies (\ref{eq:singMapsdeg}) and $H(x,w)$ satisfies
(\ref{eq:2mapsdeg}), we obtain
$$
H(z,u) = H_0(u)+H_2(u)Z^2+H_3(u)Z^3+O(Z^4),
$$
where $Z = \sqrt{1-z/\sigma}$, and $H_3(u)$ can be  computed as

$$
H_3(u) = \left({1-\sigma\over 1+\sigma}\right)^{3/2}
\left({M_3\left(u(1+\sigma)/(1+u\sigma)\right) \over 1+u\sigma}
-{M_3(1)u\sigma\over 1+\sigma}\right).
$$
The probability generating function of the distribution is given by
\begin{equation}\label{eq:prob2Maps}
p_H(u) = {H_3(u)\over H_3(1)} =
{p_M\left( \ds{u(1+\sigma)\over 1+u\sigma}\right)
\ds{1+\sigma \over 1+u\sigma} -u\sigma \over 1- \sigma},
\end{equation}
as claimed in the statement.

On the other hand,  by (\ref{eq:3mapsdeg}), $K(u,z)$ satisfies
$$
K(z,u) = K_0(u) + K_2(u) Z^2 + K_3(u) Z^3 + O(Z^4),
$$
where now $Z = \sqrt{1-z/\tau}$ and $K_3(u)$ is
$$
K_3(u) = \left({1\over 1+\tau}\right)^{3/2}
\left(H_3(u)-H_3(1){\sigma u^2}\right).
$$
The probability generating function of the distribution is given by

\begin{equation}\label{eq:prob3Maps}
p_K(u) = {K_3(u)\over K_3(1)} = {p_H(u)-u^2\sigma\over 1-\sigma}.
\end{equation}

The asymptotics of the distributions can be obtained from that of $p_M(u)$.
The singularity of $p_M(u)$ is at $u_M=6/5$, and its
expansion is computed  from the explicit formula
in (\ref{eq:distMaps}) as
\begin{equation}\label{eq:mapsexpansion}
p_M(u) = P_{-3}U^{-3} + O(U^{-2}),
\end{equation}
where $U = \sqrt{1-5u/6}$ and $P_{-3} = 1/(4\sqrt{10})$.
The singularity of $p_H$ and $p_K$ is obtained by solving
the equation
$$
\ds{u(1+\sigma)\over 1+u\sigma} = u_M = {6\over 5},
$$
giving $u_H = u_K = \sqrt{3/2}$. Hence, the
exponential growth constants are $w_H = w_K = \sqrt{2/3}$. The singular
expansion of $p_H(u)$ is obtained by composing
(\ref{eq:prob2Maps}) and (\ref{eq:mapsexpansion}), giving as a result
\begin{equation}\label{eq:2mapsexpansion}
p_H(u) = Q_{-3}U^{-3} + O(U^{-2}),
\end{equation}
where now $U = \sqrt{1-u\sqrt{2/3}}$, and
$Q_{-3} = P_{-3}\sqrt{15(1-\sigma)/8} = \sqrt{3(1-\sigma)}/16$.
The singular expansion of $p_K(u)$ is obtained by composing
(\ref{eq:prob3Maps}) and (\ref{eq:2mapsexpansion})
giving as a result
\begin{equation}\label{eq:3mapsexpansion}
p_K(u) = R_{-3}U^{-3} + O(U^{-2}),
\end{equation}
where $U$ is as before  and
$R_{-3} = Q_{-3}/(1-\sigma) = \sqrt{3/(1-\sigma)}/16$.

The estimates for $p_H(k)$ and $p_M(k)$ follow
by the transfer theorem, provided that the probability
generating functions can be extended to a $\Delta$-domain.
Since we know explicitly $p_M(u)$, we also know that
it is analytic at 
$D=\mathbb{C}\setminus(-\infty,-2]\cup[6/5,\infty)$.
By Equation (\ref{eq:prob2Maps}) we know that
if $u(1+\sigma)/(1+u\sigma)\in D$ then $p_H$ and $p_K$
are analytic at $u$. By inverting the expression
we can check that if $u(1+\sigma)/(1+u\sigma)\notin D$
then $u\in (-\infty,-1/(8-3\sqrt{6})]\cup[\sqrt{3/2},\infty)$,
and therefore $p_H$ and $p_K$ are analytic in a 
$\Delta$-domain.
\end{proof}

\section{Equations for 2-graphs and 3-graphs}\label{sec:equations}

In this section we find expressions for the generating functions of 2- and 3-graphs in terms of the generating function of connected graphs. The results are completely general and specialize to the generating functions of planar graphs, since a graph is planar if and only if its core its planar, and in turn the core is planar if and only if its kernel is planar.

Let $C(x,y)$ be the generating function of connected graphs, where $x$ marks vertices and $y$ marks edges. Denote by $H(x,y)$ and $K(x,y)$ the generating functions, respectively, of 2-graphs and 3-graphs.
We will find  equations of the form
$$
\begin{array}{ll}
H(x,y) &= C(A_1(x,y),B_1(x,y))+E_1(x,y) \\
K(x,y) &= C(A_2(x,y),B_2(x,y))+E_2(x,y),
\end{array}
$$ where
$A_i$, $B_i$ and $E_i$ are explicit functions.

From now on  all  graphs are labelled,
and all generating functions are of the exponential type.

\paragraph{2-graphs.}\label{subsec:cores}
Let $\mathcal{G}$ be a connected graph. The core
$\mathcal{C}$ of $\mathcal{G}$
is obtained by removing repeatedly vertices of degree one, so that
 $\mathcal{G}$ is obtained from $\mathcal{C}$
by replacing each vertex of $\mathcal{G}$ with a rooted
tree. The number $T_n$ of rooted trees
with $n$ edges is known to be $n^{n-1}$, and the
generating function $T(x) = \sum T_nx^n/n!$ satisfies
$$
T(x) = xe^{T(x)}.
$$
The core of $\mathcal{G}$ can be empty, in which  case
$\mathcal{G}$ must be an (unrooted) tree. The number $U_n$ of
unrooted trees is known to be $n^{n-2}$, and
the generating function $U(x) = \sum u_n x^n/n!$
is equal to
$$U(x) = T(x) - \frac{T(x)^2}{2}.
$$

\begin{theorem}\label{th:cores}
Let $h_n$ be the number of 2-graphs with $n$ vertices.
Then  $H(x) = \sum h_n x^n/n!$ is given by
\begin{equation}\label{eq:cores}
H(x) = C(xe^{-x})-x+\frac{x^2}{2}.
\end{equation}
\end{theorem}

\begin{proof}
The decomposition of a graph into its core and the attached rooted trees implies the following equation:
\begin{equation}\label{eq:corestoconnected}
    C(z) = H(T(z)) + U(z).
\end{equation}
The first summand corresponds to the case where the core is non-empty, and
the second summand corresponds to the case where the graph is a tree.
In order to invert the former relation let $x=T(z)$, so that
$$
z=xe^{-x}, \qquad U(z) = x-\frac{x^2}{2}.
$$
We obtain
$$
H(x) = C(xe^{-x})-x+\frac{x^2}{2} =
\frac{x^3}{3!}+10\frac{x^4}{4!}+252\frac{x^5}{5!}+\ldots
$$
\end{proof}

Equation~(\ref{eq:cores}) can be enriched
 by taking edges into account.
The generating functions $T(x,y)$ and $U(x,y)$ are easily
obtained as $T(x,y) = T(xy)/y$
and $U(x,y) = U(xy)/y$, and a  quick computation gives
\begin{equation}\label{eq:coresedges}
H(x,y) = C(xe^{-xy},y)-x+\frac{x^2y}{2} =
y^3\frac{x^3}{3!}
+(3y^4+6y^5+y^6)\frac{x^4}{4!}+\ldots
\end{equation}

\paragraph{3-graphs.}\label{subsec:3-graphs}
A multigraph is a graph where loops and
multiple edges are allowed. As in the case of simple
graphs, we define a $k$-multigraph as a connected multigraph in which the degree of each vertex is at least $k$. Let $\mathcal{\widetilde C}$ be a  2-multigraph.
The kernel $\mathcal{\widetilde K}$ of $\mathcal{\widetilde C}$
is defined as follows: replace every maximal path of vertices
of degree two in $\mathcal{\widetilde C}$ with a single
edge. Clearly $\mathcal{\widetilde K}$ is a  3-multigraph (unless $\mathcal{\widetilde C}$ is a cycle),
and $\mathcal{\widetilde C}$ can be recovered  by replacing
edges in $\mathcal{\widetilde K}$ with paths.

Let $\mathcal{\widetilde G}$ be a multigraph.
For each $i\geq 1$, let $\alpha_i$ be the number of vertices in
$\mathcal{\widetilde G}$ that are incident to exactly $i$ loops,
and let $\beta_i$ be the number of $i$-edges, that is,
edges of multiplicity $i$.
The weight of $\mathcal{\widetilde G}$ is defined as
$$w(\mathcal{\widetilde G}) = \prod_{i\geq 1} \left(\frac{1}{2^{i}i!}\right)^{\alpha_i}\cdot
 \prod_{i\geq 1} \left(\frac{1}{i!}\right)^{\beta_i}.
 $$
This definition is justified by the fact that when replacing
an $i$-edge with $i$ different paths, the order of the paths is irrelevant.
Similarly, when replacing a loop with a path, the orientation is irrelevant.
Note that the weight 
satisfies $0<w(\mathcal{\widetilde G})\leq 1$, and
moreover $w(\mathcal{\widetilde G}) = 1 $ if and only if
$\mathcal{\widetilde G}$ is simple.
With this definition, the sum $\widetilde K_n$ of the weights of all 3-multigraphs with $n$ vertices is finite.

As a preliminary  step to computing  the generating function of
 3-graphs, we establish  a relation between
 3-multigraphs and  connected multigraphs.
In order to distinguish between edges of different multiplicity, we  introduce  infinitely many variables as follows.
Let $\widetilde C_{n,m,l_1,l_2,\ldots}$ be the sum of the
weights of connected  multigraphs with $n$ vertices, $m$ loops
and $l_i$ $i$-edges for each $i\ge1$. Define similarly  $\widetilde K_{n,m,l_1,l_2,\ldots}$
for  3-multigraphs, and let
$$\widetilde C(x,z,y_1,y_2,\ldots) = \sum
\widetilde C_{n,m,l_1,l_2,\ldots} x^nz^my_1^{l_1}y_2^{l_2}\ldots/n!$$
and
$$\widetilde K(x,z,y_1,y_2,\ldots) = \sum
\widetilde K_{n,m,l_1,l_2,\ldots} x^nz^my_1^{l_1}y_2^{l_2}\ldots/n!.
$$

\begin{theorem}\label{th:kernel}
Let $\widetilde C(x,z,y_1,y_2,\ldots)$ and
$\widetilde K(x,z,y_1,y_2,\ldots)$ be as before. Then
\begin{equation}\label{eq:multikernels}
\begin{array}{ll}
\widetilde K (x,z,y_1,y_2,\ldots) = \\
\widetilde C\left(xe^{-x(y_1+s)},-sxy_1-xy_2+z,
s+y_0, s^2+2y_1s+y_2,\ldots,
\sum_{j=0}^k \binom{k}{j}y_js^{k-j},\ldots\right) \\
+E(x,y_1),
\end{array}
\end{equation}
where
$$y_0 = 1, \quad
s = -\frac{xy_1^2}{1+xy_1},\quad
E(x,y) = -x+\frac{x^2y}{2+2xy}-\ln \sqrt{1+xy}+\frac{xy}{2}-\frac{(xy)^2}{4}.
$$
\end{theorem}

The proof of Theorem~\ref{th:kernel} is quite technical and is given  below.
As a corollary we obtain the generating function of  3-graphs.
Recall that $C(x,y)$ is the generating function of  connected graphs.

\begin{corollary}\label{cor:kernel}
Let $K_{n,m}$ be the number of  3-graphs with
$n$ vertices and $m$ edges. The
 generating function $K(x,y) = \sum K_{n,m}x^ny^m/n!$
is given by
\begin{equation}\label{eq:simplekernel}
K(x,y) = C\left(A(x,y),B(x,y) \right)+E(x,y),
\end{equation}
where
$$A(x,y) = xe^{(x^2y^3-2xy)/(2+2xy)}, \qquad
B(x,y) =  (y+1)e^{-xy^2/(1+xy)}-1, $$
and $E(x,y)$ is as in Theorem~\ref{th:kernel}.
\end{corollary}

\begin{proof}
Since  the weight of a simple graph is one,
the number of simple 3-graphs is equal to the number
of weighted 3-multigraphs without loops or multiple edges. This
observation leads to
\begin{equation}\label{eq:kernels}
K(x,y) = \widetilde K (x,0,y,0,\ldots,0,\ldots).
\end{equation}
Moreover, for each connected multigraph $\mathcal{\widetilde G}$,
a  connected simple graph $\mathcal{G}$ can be obtained
by removing loops and replacing each multiple edge with a single edge.
Then $\mathcal{\widetilde G}$ is obtained from $\mathcal{G}$
by replacing each edge with a multiple edge, and attaching  zero or
more loops at each vertex. This can be encoded as

\begin{equation}\label{eq:multicon}
\widetilde{C}(x,z,y_1,y_2,\ldots,y_k,\ldots) =
C\left( x{e^{z/2}},\sum _{i\geq1}{\frac {y_{{i}}
}{ i !}} \right),
\end{equation}
where the exponential and the $1/i!$ terms take care
of the weights.
Finally, Equation~(\ref{eq:simplekernel}) follows 
by combining (\ref{eq:kernels}),~(\ref{eq:multikernels})
and~(\ref{eq:multicon}).
\end{proof}

We remark that a formula equivalent to (\ref{eq:simplekernel})
was obtained by Jackson and Reilly~\cite{jackson}, using the principle of inclusion and exclusion. Our approach emphasizes the assignment of weights to multigraphs, which are needed in the various combinatorial decompositions.

Note that taking $y=1$ in Equation~(\ref{eq:simplekernel})
we obtain the univariate generating function $K(x)$ of
3-graphs as

\begin{equation}\label{eq:simpleunikernel}
K(x) = K(x,1) = C(A(x,1),B(x,1))+E(x,1).
\end{equation}

The proof of Theorem~\ref{th:kernel} requires the generating
function of 2-multigraphs.
Let $\widetilde H_{n,m,l_1,l_2,\ldots}$ be the sum of the
weights of  2-multigraphs with $n$ vertices, $m$ loops
and $l_i$ $i$-edges ($i\ge1$), and let
$$\widetilde H(x,z,y_1,y_2,\ldots) = \sum
\widetilde H_{n,m,l_1,l_2,\ldots} x^nz^my_1^{l_1}y_2^{l_2}\ldots/n!.
$$

\begin{lemma}
Let $\widetilde H(x,z,y_1,y_2,\ldots)$ and
$\widetilde K(x,z,y_1,y_2,\ldots)$
be as before, and let $s = \frac{xy_1^2}{1-xy_1}$. Then the  following equation holds:
\begin{equation}\label{eq:multikernelcores}
\begin{array}{ll}
\widetilde{K}(x,z,y_1,y_2,\ldots,y_k,\ldots) = \\ 
\widetilde{H}\left(x,-sxy_1-xy_2+z,y_1+s,y_2+2y_1s+s^2,\ldots,
\sum_{j=0}^{k}\binom{k}{j}y_js^{k-j},\ldots\right)
\\ \qquad 
-\ln \sqrt{1+xy_1}-\ds\frac{xz}{2}+\frac{x^2y_2}{4}+
\frac{xy_1}{2}-\frac{(xy_1)^2}{4}.
\end{array}
\end{equation}
\end{lemma}
\begin{proof}
The kernel of a 2-multigraph is obtained by  replacing
each edge with a path. This  implies the following equation:
\begin{equation}\label{eq:multicoreskernel}
\begin{array}{ll}
\widetilde{H}(x,z,y_1,y_2,\ldots,y_k,\ldots) = \\ \widetilde{K}\left(x,sxy_1+xy_2+z,y_1+s,y_2+2y_1s+s^2,\ldots,
\sum_{j=0}^{k}\binom{k}{j}y_js^{k-j},\ldots\right)\\
\qquad 
-\ln \sqrt{1-xy_1}+\ds\frac{xz}{2}+\frac{x^2y_2}{4}-
\frac{xy_1}{2}-\frac{(xy_1)^2}{4}.
\end{array}
\end{equation}

The first summand corresponds to the case where there is
at least one vertex of degree $\ge3$, and thus the kernel
is not empty. The other summands correspond to cycles (each vertex is of degree exactly two), and  from the logarithm encoding cycles we must
take care of cycles of length one or two.

If the kernel is not empty, we replace every edge and every loop with
a path. The expression $s$ encodes  a nontrivial
path,  consisting of at least one vertex.
Each loop can be replaced with either another loop, or a vertex and a double edge, or a path consisting of at least two vertices; these operations are
encoded, respectively,  by $z,xy_2$ and $s$. Note that if the kernel has an $i$-loop, then we can replace any of the loops with a path, in both
orientations. Therefore there are $2i$ ways to obtain the
same graph, which compensates the fact that the weight
of the new graph will be $2i$ times the weight of the old graph.
Each $k$-edge can be replaced with a $j$-edge and $k-j$ nontrivial paths, where $0\leq j \leq k$. There are $(k-j)!$ ways to obtain the same graph, and the weight becomes ${k!}/{j!}$ times the previous weight. Therefore
$y_k$ is replaced with $\binom{k}{j}y_js^{k-j}$, for $j=0,\dots,k$.

A simple computation shows that   inverting (\ref{eq:multicoreskernel})
gives (\ref{eq:multikernelcores}), as claimed.
\end{proof}

\begin{proof}[Proof of Theorem~\ref{th:kernel}]
Given a multigraph it is clear
that every vertex incident to a loop or to a multiple edge
belongs to the core. Therefore,
Equation~(\ref{eq:coresedges}) can be easily extended
to multigraphs, giving the equation
\begin{equation}\label{eq:multicores}
\widetilde{H}(x,z,y_1,y_2,\ldots,y_k,\ldots) =
\widetilde C \left( x{e^{-xy_1}},z,y_1,y_2,\ldots,y_k,\ldots \right)-x
+\frac{{x}^{2}y_1}{2}.
\end{equation}
Finally, Equation~(\ref{eq:multikernels}) follows by
composing (\ref{eq:multikernelcores}) and
(\ref{eq:multicores}).
\end{proof}

As mentioned before, Theorem \ref{th:cores} and Corollary \ref{cor:kernel}
hold for planar graphs as well. In the next section we use them to enumerate and analyze planar 2- and 3-graphs.

\section{Planar graphs}\label{sec:graphs}

In this section we follow the ideas of Section~\ref{sec:maps} on planar maps in order to obtain related results for planar 2-graphs and 3-graphs.
The asymptotic enumeration of planar graphs was solved in~\cite{gn}, as well as the distribution of the number of edges.
From now on we assume that we know the generating function
$C(x,y)$ of connected planar graphs, where
$x$ marks vertices and $y$ marks edges, as well as
its main properties, such as the dominant singularities and the singular expansions around them (see \cite{gn} for details).

In this section we use the equations obtained in
Section~\ref{sec:equations} to compute several parameters
in planar graphs. Most of the computations will be analogous
to the ones of maps, but technically more involved.
In order to compare the following results, we recall \cite{gn}
that the number of connected planar graphs is
$c_n\sim \kappa n^{-7/2}\gamma^n$, where $\kappa\approx 0.4104\cdot 10^{-5}$
and $\gamma \approx 27.2269$. As expected, there are exponentially
fewer connected 2-graphs and 3-graphs than connected planar graphs.
Besides, the expected degree of 2-graphs and 3-graphs is 
larger.

\subsection{Planar 2-graphs}

We start our analysis with the enumeration of planar 2-graphs. 

\begin{theorem}\label{th:2graphs}
Let $h_n$ be the number of planar 2-graphs.
The following estimate holds:
$$
  h_n  \sim \kappa_2 n^{-7/2} \gamma_2^nn!,
$$
\noindent where $\gamma_2\approx 26.2076$ and $\kappa_2  \approx 0.3724\cdot 10^{-5}.$
\end{theorem}

\begin{proof}
Recall Equation~(\ref{eq:cores}) from Section~\ref{sec:equations}:
$$H(x) = C(xe^{-x})-x+\frac{x^2}{2}.$$
In order to obtain an asymptotic estimate for $h_n$ we need
to locate the dominant singularity of $H(x)$.
The  singularity of $C(x)$ is  $\rho = \gamma^{-1} \approx 0.0367$~\cite{gn}.
Hence the singularity of $H(x)$ is at
$\sigma = T(\rho) \approx 0.0382$. Therefore, the exponential growth constant
of $h_n$ is $\gamma_2 = \sigma^{-1} \approx 26.2076$.
Note that we use the same symbol $\sigma$ as in Section~\ref{sec:maps} for maps, but they correspond to different constants. No confusion should arise and it helps emphasizing the parallelism between planar maps and graphs.

The singular expansion of $C(x)$ at the singularity $x=\rho$ is
$$
    C(x) = C_0+C_2X^2+C_4X^4+C_5X^5+O(X^6),
$$
where $X =\sqrt{1-x/\rho}$, and $C_5 \approx -0.3880\cdot 10^{-5}$
is computed in~\cite{gn}.
Plugging this expression into (\ref{eq:cores}) and expanding gives
$$
    H(x) = H_0+H_2X^2+H_4X^4+H_5X^5+O(X^6),
$$
where now $X=\sqrt{1-x/\sigma}$ and
$H_5 = C_5 (1-\sigma)^{5/2}\approx -0.3520\cdot 10^{-5}$.
The estimate for $h_n$ follows directly by the transfer theorem,
provided that $H$ can be extended to a $\Delta$-domain.
As opposed to the case of maps, we do not have an exact
expression for $C$, and because of
the relation of Equation (\ref{eq:cores}), it is not enough
to assume that $C$ can be extended to a
$\Delta$-domain, since $|(-\sigma)\exp(-(-\sigma))| > \rho$.
Instead, we use an alternative expression for $H$. 

Define $A(x)$ as the generating function
of connected planar graphs with an unlabelled root vertex
where all the vertices except, perhaps, the root,
have degree at least 2. If the root has degree 2 then 
graphs in $A$ are encoded by $H'(x)$. Otherwise either the graph is reduced to a single vertex or the root is connected to a rooted 2-graph through a path of arbitrary length and they are encoded by $\frac{x}{1-x}H'(x)$. Hence we have 
\begin{equation}\label{eq:AH}
A(x) = {H'(x)\over 1-x}+1.
\end{equation}

Let $B(x)$ be the generating of planar 2-connected graphs. The unique decomposition of a rooted connected graph into blocks is reflected  (see \cite{gn}) into the basic equation
$C'(x) = \exp\left( B'(xC'(x))\right)$. 
The radius of convergence $R$ of $B$ is given by $R=\rho C'(\rho)$, and $R$ is the only singularity in the circle of convergence of $B(x)$. 

A straightforward modification including paths as building blocks in the decomposition gives
\begin{equation}\label{eq:AB}
A(x) = \exp(B'(xA(x))-x).
\end{equation}
Let $F(x) = \exp(B'(xA(x))-x)$ be the right-hand side of (\ref{eq:AB}). Equation (\ref{eq:AH}) shows 
that $A$ has the same singularities
as $H$ in the open ball of radius 1. We now use (\ref{eq:AB})
to prove that  $A$, and therefore $H$, can be extended to a $\Delta$-domain.

The proof has two parts. First we have to prove
that $A$ behaves like a square root near its singularity
$x=\sigma$. This follows from~\cite[Theorem 2.31]{drmotatrees}, using $r(x)=R/x$ (in the notation of \cite{drmotatrees}).
Then we need to prove that there is no branch
point when solving  $A=F(A,x)$ for $x$ in the circle of convergence $|x|=\sigma$. Since $F_A(A,x)=xAB''(xA)$ is
a positive function, and $F_A(A(\sigma),\sigma)=
RB''(R)<1$, we have that $|F_A(A(x),x)|<1$,
so it is analytic in a neighbourhood of $x$.
By compactness   $A$ is analytic in a  $\Delta$-domain at $\sigma$.
\end{proof}

Our next result is a limit law for the number of edges in a random
planar 2-graph. We recall \cite{gn} that the expected number of edges in random connected planar graphs is asymptotically $\mu n$, where $\mu \approx 2.2133$,
and the variance is $\lambda n$ with $\lambda\approx 0.4303.$

\begin{theorem}
The number $X_n$ of edges in a random planar 2-graph with $n$ vertices
is asymptotically Gaussian with
$$
\ex X_n \sim \mu_2n
\approx 2.2614n,
$$
$$
\var X_n \sim \lambda_2n
\approx 0.3843n.
$$
\end{theorem}

\begin{proof}
Equation~(\ref{eq:coresedges}) from Section~\ref{sec:equations}
$$H(x,y) = C(xe^{-xy},y)-x+\frac{x^2y}{2}$$
implies that the singularity $\sigma(y)$ of the univariate
function $x\mapsto H(x,y)$ is given by
$$\sigma(y)e^{-\sigma(y)y} = \rho(y),$$

\noindent where $\rho(y)$ is the singularity of the univariate
function $x\mapsto C(x,y)$.
An easy calculation gives
$$\mu_2=-\ds\frac{\sigma'(1)}{\sigma(1)} =
{-\rho'(1)/\rho - \sigma \over 1-\sigma} = {\mu - \sigma \over 1-\sigma} \approx 2.2614,
$$
which provides the constant for the expectation.
Similarly
$$
\renewcommand{\arraystretch}{3}
\begin{array}{ll}
& \lambda_2=-\ds{\sigma''(1) \over \sigma(1)} -{\sigma'(1)\over\sigma(1)} +
\left({\sigma'(1)\over\sigma(1)}\right)^2 = \\
&  \ds{\ds{-\rho''(1)\over \rho(1)}-3\sigma'(1)
-{3\sigma'(1)^2\over \sigma}+\sigma'(1)^2+2\sigma'(1)\sigma
+\sigma^2-{\sigma'(1)\over\sigma}+\left({\sigma'(1)\over \sigma}\right)^2 \over 1-\sigma}.
\end{array}
$$
This value can be computed from the known values of $\mu,\lambda$ and $\sigma$.

Again, in order to apply the quasi-powers theorem we  need
to prove that $H(x,y)$ is   $\Delta$-analytic
for $y$ close enough to 1. Define $A(x,y)$
as the generating function of connected planar graphs
with an unlabelled root where all the vertices except the root have degree at least 2.
The following equations are a direct extension of (\ref{eq:AH}) and (\ref{eq:AB}): 
$$
A(x,y)={H_x(x,y)\over 1-xy}+1,
$$
$$
A(x,y)=\exp(B_x(xA(x,y),y)-xy)=F(A,x,y).
$$
From the first equation we know that $A$ and $H$
have the same singularities for $x$, $y$ such
that $xy<1$, so we just need to prove that
for values $y_0$ near $1$ the function  $A(x,y_0)$
is  $\Delta$-analytic.
The proof is analogous to that of 
Theorem~\ref{th:2graphs}. First,  $A(x,y)$ behaves like a square root near the singularity $\sigma(y_0)$, again by~\cite[Theorem 2.31]{drmotatrees}
taking $r(x,u)=R(u)/x$.
Then we need that, when $|x|=R(y)$,  
$F_A(A(x,y),x,y)\neq 1$ holds. Since  
$F_A$ is positive,
$F_A(A(x,1),x,1)<1$, and since both $F$ and $A$
are continuous in $y$, for values of $y$
close enough to 1 the inequality  holds,
so again we can extend $A(x,y)$
to a $\Delta$-domain at $\sigma(y)$. 
\end{proof}

Next we determine a limit law for the size of the core  in random connected planar graphs.

\begin{theorem}\label{th:sizecore}
The size $X_n$ of the core of a random connected planar graph
with $n$ edges is asymptotically Gaussian with
$$
 \ex X_n \sim (1-\sigma)n \approx 0.9618n,
\qquad \var X_n \sim \sigma n \approx 0.0382n.
$$
\end{theorem}

\begin{proof}
The generating function $\widehat C(x,u)$ of connected planar graphs, where $u$ marks the size of the core, is given by
\begin{equation}
\widehat C(x,u) = H(uT(x)) + U(x).
\end{equation}
It follows that the singularity $\xi(u)$ of the univariate
function $x\mapsto \widehat C(x,u)$ is given by the equation
$$uT(\xi(u)) = \sigma.$$
We can isolate $\xi(u)$ obtaining the explicit formula
$$\xi(u) = {\sigma e^{-\sigma / u} \over u}.$$
An easy calculation gives
$$-\frac{\xi'(1)}{\xi(1)} = 1-\sigma,\qquad
-{\xi''(1)\over \xi(1)} -{\xi'(1)\over \xi(1)}
+\left({\xi'(1)\over \xi(1)}\right)^2= \sigma.$$
In order to apply the quasi-powers theorem
we need to show that, for $u_0$ close enough
to 1 we can extend the generating function
$C(x,u_0)$ to a $\Delta$-domain.
As in the proof of Theorem~\ref{th:2graphs}, two steps are needed. First, we have to prove
that $C(x,u)$ is analytic near $x=\rho(u)$
if $\arg(x/\rho(u)-1)>\alpha$.
We know that this is the case for $H(x)$ near
$\sigma$, for some angle $\beta$.
Since $uT(x)$ is analytic, it is conformal and 
preserves angles locally, hence
for $u$ close enough to 1  and $x$ close enough
to $\rho(u)$, if $\arg(x/\xi(u)-1)>\alpha$ for 
some $\alpha > \beta$, then 
$uT(x)$ is close to $\sigma$
and $\arg(T(x)u/\sigma-1)>\beta$. Then $T(x)u$
is in the region of convergence of $H$ and
$C(x,u)$ is analytic in $x$.
On the other hand, if $u=1$ then $uT(x)$
is a positive function, hence if $|x|=\xi(1)$
but $x\neq \xi(1)$ then $|T(x)|<\sigma$.
This implies that if $u$ is close enough to 1
and $|x|=|\xi(u)|$ but far enough from 
$\xi(u)$,
then $|uT(x)u|<\sigma$  by the continuity
of $uT(x)u$, so $C(x,u)$ is analytic in
a neighbourhood of $x$. By compactness,
a finite number of neighbourhoods is enough,
and their union gives a $\Delta$-domain in which
$C(x,u)$ is analytic.
\end{proof}

Our next goal is to analyze the size of the trees attached
to the core of a random connected planar graph.

\begin{theorem}\label{th:numtrees}
Let $k$ be fixed and let $X_{n,k}$ count trees with $k$ vertices attached to the core
of a random connected planar graph with $n$ vertices. Then
$X_{n,k}$ is asymptotically normal and
$$
\ex X_{n,k} \sim \alpha_k n, \qquad
\var X_n \sim \beta_k n,
$$
where
$$
\alpha_k = \frac{1-\sigma}{\sigma}\frac{k^{k-1}}{k!}\rho^k,
$$
and $\beta_k$ is described in the proof.
\end{theorem}

\begin{proof}
The generating function of trees where variable $w$ marks
trees with $k$ vertices is equal to
$$
T(x,w) = T(x) + (w-1)T_kx^k,
$$
where $T_k = k^{k-1}/k!$ is the $k$-th coefficient of $T(x)$.
The composition scheme for the core decomposition is then
$$
C(x,w) = H(T(x,w)) +U(x).
$$
It follows that the singularity $\rho_k(w)$ of the
univariate function $x\mapsto C(x,w)$ is given
by the equation
$$
T(\rho_k(w)) + (w-1)T_k(\rho_k(w))^k = \sigma.
$$
An easy calculation gives
$$
\alpha_k = -\frac{\rho'_k(1)}{\rho_k(1)} =
\frac{1-\sigma}{\sigma}\frac{k^{k-1}}{k!}\rho^k
$$
$$
\beta_k = -{\rho_k''(1)\over \rho_k(1)}-{\rho_k'(1)\over \rho_k(1)}
+\left({\rho'(1)\over\rho(1)}\right)^2=
{1\over \sigma^2}\left(T_k\rho^k
(T_k\rho^k(1-2k+4\sigma-2k\sigma^2)+\sigma-\sigma^2\right)
$$
The proof that $C(x,w)$ can be extended analytically to 
a $\Delta$-domain is analogous to the proof of Theorem~\ref{th:sizecore}.
\end{proof}

As expected, $\sum_{k\geq 0}\alpha_k = 1-\sigma$,
since there are $\sigma n$ vertices not in the core, and therefore
there are $(1-\sigma)n$ trees attached to the core.
Moreover, $\sum_{k\geq 0}k\alpha_k = 1$, since a connected graph
is the union of the trees attached to its core.

To conclude this section, we consider the parameter $L_n$ equal to the size of largest tree attached to the core of a random planar connected graph.

\begin{theorem}\label{th:finalmaxtreesgraphs}
Let $L_n$ be the size of largest tree attached to the core of a random planar connected graph. Then
$$
{L_n \over \log n} \rightarrow {1\over \log(1/(e\rho))}\approx 0.4340
\qquad \text{in probability},
$$
and
$$
\mathbb{E}L_n \sim {1\over \log(1/(e\rho))} \log n
\qquad (n\rightarrow \infty).
$$
\end{theorem}


\begin{proof}
The main idea in the proof is to generalize  Theorem \ref{th:master}, assigning a numerical ``label'' $\nu$ to each vertex instead of its vertex degree. Given the same hypothesis  in the behaviour of this parameter, 
the conclusion still holds and we obtain an estimate on the maximum label. 

In our case the label is the size of the tree attached to the core
that contains the  given vertex. If the graph is itself a tree
then all  labels are equal to 0 by convention.
Therefore, in the rewording of  Theorem \ref{th:master}, 
$d_{n,k}$ denotes the probability that
a randomly selected vertex of a random planar
graph of size $n$ has label $k$, and 
$d_{n,k,l}$ denotes the probability that
two different (ordered) randomly selected 
vertices have labels $k$ and $l$. In order to compute
such probabilities we define the generating functions
$\widehat{C}(x,z)$ and $\widehat{C}(x,z,w)$ as follows:
$\widehat{C}(x,z)$ is for connected
planar graphs with a root vertex, where $x$ marks vertices
and $z$ marks the label of the root. Analogously,
$\widehat{C}(x,z,w)$ is for connected
planar graphs with two different ordered root vertices,
where $x$ marks vertices, $z$ marks the label of the first root,
and $w$  the label of the second root. 

Given a generating function $F(x)$ of labelled graphs, we let $F^\bullet(x)=xF'(x)$, which encodes graphs rooted at a vertex. 
Also, $F^{\bullet\bullet}(x)$ encodes graphs rooted at two different vertices. 
The next equation is derived from  (\ref{eq:corestoconnected})  by differentiation
$$C^\bullet(x) = H'(T(x) T^\bullet (x) + T(x),
$$ 
and the following relations extend the previous equation, marking the labels of the root vertices: 
$$
\begin{array}{ll}
\widehat{C}(x,z) &= H'(T(x))T^\bullet(zx)+T(x), \\
\widehat{C}(x,z,w) &=
H''(T(x))T^\bullet(zx)T^\bullet(wx)+
H'(T(x))T^{\bullet\bullet}(zwx)+T^\bullet(x). \\
\end{array}
$$
We then have
$$
d_{n,k} = 
{[x^nz^k]\widehat{C}(x,z)\over
[x^n]\widehat{C}(x,1)}, \qquad 
d_{n,k,l} = 
{[x^nz^kw^l]\widehat{C}(x,z,w)\over
[x^n]\widehat{C}(x,1,1)}.
$$
Also note that $\widehat{C}(x,1) = C^\bullet(x)$,
and $\widehat{C}(x,1,1)=C^{\bullet\bullet}(x)$,
which are well-known functions.
Next we verify that all the conditions in Theorem \ref{th:master} hold. 

\smallskip
\noindent
\textit{Condition 1.}
Define
$\alpha_k$ as in Theorem~\ref{th:numtrees}. Then
$$
\overline d_k = k\cdot \alpha_k = {1-\sigma\over \sigma}
{k^k\over k!}\rho^k.
$$
And one easily checks that $\log \overline d_k \sim
k \log (e\rho)$ as $k\to\infty$, as required.

\smallskip
\noindent
\textit{Condition 2.}
To check this condition we cannot use the quasi-powers theorem,
since it only proves the desired result for fixed~$k$. Since we
only need the result for $k$ tending to infinity, we can
dismiss the graphs whose core is empty. Therefore, for $k \to \infty$, 
$$
[x^nz^k]\widehat{C}(x,z) \sim 
[x^n]H'(T(x))[z^k]T^\bullet(xz)=
[x^{n-k}]H'(T(x))\cdot [z^k]T^\bullet(z).
$$
From this we obtain
$$
d_{n,k}\sim{[x^{n-k}]H'(T(x))\over
[x^n]C^\bullet(x)}\cdot [z^k]T^\bullet(z)
\sim {1-\sigma\over \sigma} \left({n-k\over n}\right)^{-5/2}\rho^k\cdot {1\over \sqrt{2\pi k}} e^k.
$$
Finally, when $k\le C \log n$ we have 
$\left({n-k\over n}\right)^{-5/2}\rightarrow 1$ and thus $d_{n,k}\sim \overline{d}_k$.

Now we have to prove a similar estimate for $d_{n,k,l}$. Let $\widehat{C}(x,z,w)= S_1 + S_2 + S_3$, where 
$$
S_1=H''(T(x))T^\bullet(zx)T^\bullet(wx), \quad S_2=
H'(T(x))T^{\bullet\bullet}(zwx), \quad S_3=T^\bullet(x).
$$
We know that the coefficients of $S_3$
are 0 when $k$ and $l$ tend to infinity.
Since we differentiate $H$ once instead of twice, it follows that  $[x^nz^kw^k]S_2
=O((k/n) [x^nz^kw^k]S_1)$. Since $k=O(\log n)$, the coefficients of $S_2$ are asymptotically smaller than those of  $S_1$. Therefore, the main asymptotic part comes from $S_1$.
We have
$$
[x^nz^kw^l]S_1(x,z,w)
=[x^{n-k-l}]H''(T(x))\cdot [z^k]T^\bullet(z)\cdot [w^l]T^\bullet(w).
$$
Then
$$
d_{n,k,l}={[x^{n-k-l}]H''(T(x))\over
[x^n]C^{\bullet\bullet}(x)}\cdot [z^k]T^\bullet(z)\cdot [w^l]T^\bullet(w)
\cdot (1+o(1))
$$
$$\sim \left({1-\sigma\over \sigma}\right)^2
\left({n-k-l\over n}\right)^{-3/2}\rho^{k+l}\cdot \left({1\over \sqrt{2\pi k}}\right)^2 e^{k+l}.
$$
When $k,\ell= O(\log n)$ we have 
$\left({n-k-l\over n}\right)^{-3/2}\rightarrow 1$, and thus
$d_{n,k,l}\sim \overline{d}_k\overline{d}_l$, as required. 

\smallskip\noindent
\textit{Condition 3.} 
We already proved that,
for $k$, $l \geq 1$, and uniformly
for any $k$, $l$, $n$,
we have 
$$[x^nz^k]\widehat{C}(x,z) = f(n, k) (e\rho)^k, \quad 
[x^nz^kw^l]\widehat{C}(x,z,w) = g(n, k, l) (e\rho)^{k+l},
$$
where $f$ and $g$ are subexponential functions, so for any
$\overline q > e\rho$ we have that
$d_{n,k}=O(\overline q^k)$
and $d_{n,k,l}=O(\overline q^{k+l})$.

Thus  Theorem \ref{th:master} applies and we conclude the proof. 
\end{proof}


\paragraph{Remark.}  
A similar result can be proved for random maps. 
Let $L_n$ be the size of largest tree attached to the core of a random rooted map with $n$ edges. Then it can be shown that 
$$
{L_n \over \log n} \rightarrow {1\over \log(3)}\approx 0.912 
\qquad \text{in probability},
$$
and
$$
\mathbb{E}L_n \sim {1\over \log(3)} \log n
\qquad (n\rightarrow \infty).
$$
The proof is similar to the prove of the previous result and we omit it for the sake of brevity.

\subsection{Planar 3-graphs}

We recall again that the generating function of connected planar graphs $C(x,y)$, where $x$ marks vertices and $y$ marks edges, was computed in~\cite{gn}.

\begin{theorem}\label{th:3graphs}
Let $k_n$ be the number of planar 3-graphs.
The following estimate holds:
$$
  k_n  \sim \kappa_3 n^{-7/2} \gamma_3^{n}n!,
$$
where
$$
\gamma_3 \approx 21.3102, \qquad \kappa_3 \approx 0.3107\cdot 10^{-5}.
$$
\end{theorem}

\begin{proof}

Recall Equation~(\ref{eq:simpleunikernel}) from Section~\ref{sec:equations}:
\begin{equation}\label{eq:kernelrep}
K(x) = C\left(A(x),B(x) \right)+E(x),
	\end{equation}
where $A(x), B(x)$ and $E(x)$ are explicit functions. 
In order to obtain an estimate for $k_n$ we need
to locate the dominant singularity of $K(x)$.
The singularities of $C(x,y)$ is given by
$(X(t),Y(t))$, where $t\in(0,1)$ and $X$, $Y$ are
explicit functions defined in~\cite{gn}.
Hence the singularity $\tau$ of $K(x)$ is obtained by solving
$$
X(t) = A(\tau), \qquad  Y(t) = B(\tau).
$$
The smallest positive solution $\tau$ of the system can be computed numerically
and is $\tau \approx 0.0469$. The exponential growth constant
is then $\gamma_3 = \tau^{-1} \approx 21.3102$.

The singular expansion of $C(x,y)$
at the singularity $x=\rho(y)$ is of the form
$$C(x,y) = C_0(y)+C_2(y)X^2+C_4(y)X^4+C_5(y)X^5+O(X^6),$$
where $X = \sqrt{1-x/\rho(y)}$, and $C_5(y)$ is
an explicit expression computed in \cite{gn}. Plugging this expression
into (\ref{eq:kernelrep}) and expanding gives
\begin{equation}\label{eq:Ksing}
K(z) = K_0 + K_2Z^2 + K_4Z^4+K_5Z^5+O(Z^6),
\end{equation}
where  $Z = \sqrt{1-z/\tau}$.
In order to compute the dominant coefficient $K_5$, we need to expand $C_5(B(z))\left(1-D(z)\right)^{5/2}$, where
$D(z)= A(z)/\rho(B(z))$, at $z = \tau$. Consider the first-order Taylor expansion of $D(z)$:
$$
D(z) = D(\tau) + D'(\tau)(z-\tau) + O((z-\tau)^2).
$$
Since $(A(\tau),B(\tau))$ is a singular point of $C(x,y)$, we have
$$A(\tau) = \rho(B(\tau)), \qquad 
D(\tau) = 	\frac{A(\tau)}{\rho(B(\tau))} = 1.
$$
 Therefore,
$\sqrt{1-D(z)}$ is computed as
$$\sqrt{\tau D'(\tau)(1-z/\tau) + O((x-\tau)^2) } =
\sqrt{\tau D'(\tau)}Z + O(Z^2),
 $$
 hence $(1-D(z))^{5/2} = (\tau D'(\tau))^{5/2}Z^5+O(Z^6)$.
Since $C_5(y)$ is analytic at $y=B(\tau)$, we conclude that
$K_5 = C_5(B(\tau))(\tau D'(\tau))^{5/2}
\approx-0.2937\cdot 10^{-5}$.
The estimate for $k_n$ follows directly by the transfer theorem, with $\kappa_3 = K_5/\Gamma(-5/2) \approx 0.3107\cdot 10^{-5}$,
provided that $K$ can be analytically extended to a 
$\Delta$-domain at $\tau$. The proof is more technical than the previous proofs of $\Delta$-analyticity and is shown in the appendix. 
\end{proof}

Our next result is a limit law for the number of edges
in a random planar 3-graph.

\begin{theorem}
The number $X_n$ of edges in a random planar 3-graph with $n$ vertices
is asymptotically Gaussian with
$$
\ex X_n \sim \mu_3 n \approx 2.4065n, \qquad
\var X_n \sim \lambda_3 n \approx 0.3126n.
$$
\end{theorem}

\begin{proof}
Recall Equation~(\ref{eq:simplekernel}) from Section~\ref{sec:equations}:
\begin{equation}
K(x,y) = C\left(A(x,y),B(x,y) \right)+E(x,y),
\end{equation}
where
$$A(x,y) = xe^{(x^2y^3-2xy)/(2+2xy)}, \qquad
B(x,y) =  (y+1)e^{-xy^2/(1+xy)}-1, $$
$$
E(x,y) = -x+\frac{x^2y}{2+2xy}-\ln \sqrt{1+xy}+\frac{xy}{2}-\frac{(xy)^2}{4}.
$$
It follows that the singularity $\tau(y)$ of the univariate
function $x\mapsto K(x,y)$ is given by the equation
$$
A(\tau(y),y) = \rho(B(\tau(y),y)),
$$
where $\rho(y)$ is as before the singularity of  $x\mapsto C(x,y)$. The value of $\tau(1)=\tau$ is already known. In order to compute $\tau'(1)$ we
differentiate and obtain
$$
A_x(\tau,1)\tau'(1) +A_y(\tau,1) =
\rho'(B(\tau,1))\left[B_x(\tau,1)\tau'(1)+B_y(\tau,1)\right].
$$
Solving for  $\tau'(1)$ we obtain
$$
\tau'(1) = -{\rho'(B(\tau,1))B_y(\tau,1)-A_y(\tau,1)\over
\rho'(B(\tau,1))B_x(\tau,1)-A_x(\tau,1)}.
$$
Since $\rho = X\circ Y^{-1}$, where $X$ and $Y$ are explicit
functions defined in \cite{gn}, $\rho'(y)$ can be computed
as $X'(Y^{-1}(y))/Y'(Y^{-1}(y))$.
After some calculations we finally get a value of
$\tau'(1)\approx-0.1129$ and
$$
 \mu_3 =-\frac{\tau'(1)}{\tau(1)} \approx 2.4065.
$$
Using the same procedure we can isolate $\tau''(1)\approx0.3700$ and obtain
$$
\lambda_3 = -\frac{\tau''(1)}{\tau(1)}
-\frac{\tau'(1)}{\tau(1)}+
\left(\frac{\tau'(1)}{\tau(1)}\right)^2\approx 0.3126.
$$
In order to apply quasi-powers theorem we 
have to show that $K(x,y)$ is analytic in
a $\Delta$-domain for $y$ close enough to 1.
The proof is a direct extension  of  that of  the Lemma in the Appendix 
by adding  variable $y$ marking edges, and we omit it to avoid repetition. 
\end{proof}

Next we determine the limit law for the size of
the kernel in random planar 2-graphs.

\begin{theorem}
The size $Y_n$ of the kernel of a random planar 2-graph
with $n$ edges is asymptotically Gaussian with
\begin{equation}
\ex Y_n \sim \mu_K n
\approx 0.8259n,\qquad
\var Y_n \sim \lambda_K n \approx 0.1205n
\end{equation}
\end{theorem}

\begin{proof}
Recall that the decomposition of a simple 2-graph into its kernel gives
$$
\begin{array}{ll}
H(x) = \widetilde H(x,0,1,0,\ldots)
\\ 
= \widetilde K\left(x,\frac{x^2}{1-x},\frac{1}{1-x},
\ldots,k\left(\frac{x}{1-x}\right)^{k-1}+
\left(\frac{x}{1-x}\right)^{k},\ldots\right)+E(x,1).
\end{array}
$$
If $u$ marks the size of the kernel then
$$H(x,u) = \widetilde K\left(ux,\frac{x^2}{1-x},\frac{1}{1-x},
\ldots,k\left(\frac{x}{1-x}\right)^{k-1}+
\left(\frac{x}{1-x}\right)^{k},\ldots\right)+E(x,1).$$
Composing with Equations~(\ref{eq:multikernels}) and (\ref{eq:multicon})
we get
$$H(x,u) = C\left(A(x,u),
B(x,u)
\right)+F(x,u)
$$
where
$$
A(x,u)=ux\exp\left({-x
\left(2u+x+{u}^{2}x-2\,ux \right)
\over 2(1-x+ux)}\right),
$$
$$
B(x,u)=-1+2\exp\left({x
\left( 1-u \right)\over1-x+ux}\right),
$$
and $F(x,u)$ is a correction term which does not affect the singular analysis. It follows that the singularity $\chi(u)$ of the univariate function $x\mapsto H(x,u)$ is given by the equation
$$
A(\chi(u),u)=\rho(B(\chi(u),u)),
$$
If we differentiate the former expression and replace u with 1 we get
$$
A_x(\sigma,1)\chi'(1)+A_y(\sigma,1)=
\rho'(1)(B_x(\sigma,1)\chi'(1)+B_y(\sigma,1)).
$$
Note that $\chi(1)=\sigma$, where $\sigma$ is, as before,
the singularity of the generating function $H(x)$ of planar
2-graphs. Moreover, $B(x,1)=1$.
After some calculations we finally get $\chi'(1)\approx-0.03135$ and
$$\mu_K = -\frac{\chi'(1)}{\chi(1)} =
\frac{2\rho'(1)e^{\sigma}+\sigma^2-\sigma+1}{1-\sigma}.$$
This is computed using the known values of $\sigma$ and $\rho'(1)=-\rho\mu$.
Using the same procedure we can isolate $\chi''(1)\approx0.05295$ and compute
$\lambda_K$ as
$$
\lambda_K = -{\chi''(1)\over \chi(1)}
-{\chi'(1)\over \chi(1)} + \left({\chi'(1)\over \chi(1)}\right)^2
\approx 0.1205.
$$
We  need to show that $H(x,u)$ is analytic in a
$\Delta$-domain. If $u=1$ we already know it for 
$H(x,1)$.
Since $A(x,u)$ and $B(x,u)$ are
both analytic, and $A(\sigma,1)=\rho$ and $B(x,1)=1$,
then for $u$ close enough to $1$ and $x$ close enough
to $\chi(u)$, by continuity, if $\arg(x/\chi(u)-1)>\alpha$
then $\arg(A(x,u)/\rho(B(x,u))-1)>\beta$ for some $\beta>0$,
as in the proof of Theorem~\ref{th:sizecore}.
Also, if  $|x|=\sigma$ but $x\neq \sigma$,
then we know that $H(x,1)$ is analytic near $x$.
Again by continuity, if $u$ is close enough to 1
then $H(x,u)$ is analytic at $(x,u)$,
and by compactness this is sufficient to prove analyticity in a $\Delta$-domain.
\end{proof}

Note that, since the expected size of the core of a random connected
planar graph is $1-\sigma$, the expected size of the kernel of a random
connected planar graph with $n$ vertices is
asymptotically
$(1-\sigma)\mu_K n = (2\rho'(1)e^{\sigma}+\sigma^2-\sigma+1)n \approx 0.7944n$.

\section{Degree distribution}\label{sec:degree}

In this section we compute the limit probability that a vertex
of a planar 2-graph or 3-graph has a given degree.
In order to do that, we compute
the probability distribution of the root of a rooted planar 2-graph
and 3-graph. Since every vertex is equally likely to be the root,
we conclude that the average distribution is the same. Note that
this is not true for maps, so in this section we only compute
the distribution for graphs.
This section is rather technical, especially the part of 3-graphs,
so that is why we separate its content from that of Section~\ref{sec:graphs}.

Let $c^{\bullet}_n$ be the number of rooted connected planar graphs
with $n$ vertices, i.e., $c^{\bullet}_n = n\cdot c_n$.
Let $C^{\bullet}(x) = \sum c^{\bullet}_n x^n = xC'(x)$ be its associated
generating function.
Let $c^{\bullet}_{n,k}$ be the number of rooted connected planar graphs
with $n$ vertices and such that the root degree is exactly $k$.
Let $C^{\bullet}(x,w) = \sum c^{\bullet}_{n,m} x^n u^m$ be its associated
generating function. The limit probability $d_k$ that the root vertex has
degree $k$ can be obtained as
$$
d_k = \lim_{n\rightarrow \infty}{c^{\bullet}_{n,k}\over c^{\bullet}_n}=
\lim_{n\rightarrow \infty}{[x^n][w^k]C^{\bullet}(x,w)\over [x^n]C^{\bullet}(x)}.
$$
Therefore, the probability distribution $p(w) =\sum d_k w^k$ can be
obtained from the knowledge of $C^{\bullet}(w,u)$. In \cite{degree} this
function is computed, and $d_k$ is proven to be asimptotically
$$
d_k \sim c\cdot k^{-1/2}q^k,
$$
where $c\approx 3.0175$ and $q\approx 0.6735$ are computable constants.
Our goal is to obtain similar results for 2-graphs and 3-graphs, by
respectively computing generating function $H^{\bullet}(x,w)$
and $K^{\bullet}(x,w)$ in terms of $C^{\bullet}(x,w)$.

\subsection{2-graphs}

\begin{theorem}\label{th:coresdeg}
Let $h^{\bullet}_{n,k}$ be the number of rooted 2-graphs with
$n$ vertices and with root degree $k$. Let
$H^{\bullet}(x,w) = \sum h^{\bullet}_{n,k} x^n w^k$ be its associated
generating function. The following equation holds
\begin{equation}\label{eq:coresdeg}
H^{\bullet}(x,w) = e^{x(1-w)}C^{\bullet}(xe^{-x},w)
-xwC^{\bullet}(xe^{-x})-x+x^2w
\end{equation}
\end{theorem}
\begin{proof}
The decomposition of a graph into ins core and the attached rooted trees
implies the following equation:
$$
C^{\bullet}(z,w) = H^{\bullet}(T(z),w){T(z,w)\over T(z)}+
H^\bullet(T(z)){wT(z,w)\over 1-T(z)}+T(z,w),
$$
where $T(z,w) = z\cdot e^{wT(z)}$ is the generating function of rooted trees where
$w$ marks the degree of the root.
The first addend corresponds to the case where the root is in the core.
In this case, the degree of the graph root is the degree of the core root
plus the degree of the root of its appended tree. The second addend
corresponds to the case where the root is in an attached tree.
In this case there is a sequence of trees between the core and the
root, and finally a rooted tree. The degree of the graph root
is the degree of the root of the rooted tree plus one. The last addend
corresponds to the case where the graph is a tree, and therefore
its core is empty.

In order to invert the former relation let $x = T(z)$ so that
$$
z = xe^{-x},\quad T(z,w) = x e^{-x(1-w)},\quad
H^{\bullet}(T(z)) = (1-x)C^{\bullet}(xe^{-x})+x^2-x.
$$
After some calculations we obtain
$$
H^{\bullet}(x,w) = e^{x(1-w)}C^{\bullet}(xe^{-x},w)
-xwC^{\bullet}(xe^{-x})-x+x^2w =
$$
$$
={1\over 2}w^2x^3+\left(w^2+{2\over 3}w^3\right)x^4+
\left({9\over 2}w^2+{13\over 3}w^3+{41\over 24}w^4 \right)x^5+\ldots
$$
\end{proof}

The probability distribution $p(w)$ can be computed using transfer theorems.
The expansion of $C^\bullet (x,w)$ near the singularity $x=\rho$ gives
the following equation
\begin{equation}\label{eq:condeg}
C^\bullet (x,w) = C_0(w) + C_2(w)X^2 + C_3(w)X^3 + O(X^4),
\end{equation}
where $X = \sqrt{1-x/\rho}$. The probability
distribution can be computed as
$$
p(w) = {C_3(w)\over C_3(1)}.
$$
Our goal is to obtain the same result by applying the relation
obtained in (\ref{eq:coresdeg}).

\begin{theorem}
Let $e_k$ be the limit probability that a random vertex has degree $k$
in a 2-graph. Let $p_H(w) = \sum e_k w^k$ be its probability distribution.
Let $p(x)$ be as before. The following equation holds:
\begin{equation}\label{eq:distcores}
p_H(w) = {e^{\sigma(1-w)}p(w)-\sigma w \over 1-\sigma},
\end{equation}
where $\sigma = T(\rho)$, as in Theorem~\ref{th:2graphs}.
Furthermore, the limiting probability that the degree of
a random vertex is equal to $k$ exists, and is asymptotically
$$
p_H(k)\sim \nu_2k^{-1/2}q^k,
$$
where $q\approx 0.6735$ and $\nu_2\approx3.0797$.
\end{theorem}

\begin{proof}
Since $C^\bullet(x,w)$ satisfies (\ref{eq:condeg}), and
$H^\bullet(x,w)$ satisfies (\ref{eq:coresdeg}), we obtain
$$
H^\bullet(x,w) = H_0(w) + H_2(w)X^2 + H_3(w)X^3+O(X^4),
$$
where $X = \sqrt{1-x/\sigma}$, and $H_3(w)$ is computed as
$$
H_3(w) = e^{\sigma(1-w)}C_3(w)(1-\sigma)^{3/2}-
w\sigma C_3(1) (1-\sigma)^{3/2}
$$
The probability generating function of the distribution
is given by
$$
p_H(w) = {H_3(w)\over H_3(1)} =
{(1-\sigma)^{3/2}\left(e^{\sigma(1-w)}C_3(w) - w\sigma C_3(1)\right)
\over (1-\sigma)^{3/2}C_3(1)(1-\sigma)}
= {e^{\sigma(1-w)}p(w)-\sigma w \over 1-\sigma}.
$$

The asymptotics of the distribution can be obtained from $p(w)$.
The singularity of $p(w)$ is obtained in \cite{degree} as
$r \approx 1.4849$. The expansion of $p(w)$ near the singularity
is computed as
$$
p(w) = P_{-1}W^{-1} + O(1),
$$
where $P_{-1}\approx 5.3484$ is a computable constant,
and $W=\sqrt{1-w/r}$. Plugging this expression into (\ref{eq:distcores})
we get
$$
p_H(w) = Q_{-1}W^{-1} + O(1),
$$
where $Q_{-1} = P_{-1}e^{\sigma(1-r)}/(1-\sigma) \approx 5.4586$.
The estimate for $p_H(k)$ follows directly by singularity analysis.
\end{proof}

\subsection{3-graphs}

In order to prove a similar result for 3-graphs, we need to extend
the generating function $C^{\bullet}(x.w)$ so that it takes edges
into account. This function $C^{\bullet}(x,y,w)$
was computed in \cite{degree}, and our goal is to obtain the
analogous generating function for 3-graphs, $K^{\bullet}(x,w)$,
in terms of 
We remark that the expression given in \cite{degree} for 
$C^{\bullet}(x,y,w)$ is extremely involved and needs several pages to write it down.

\begin{theorem}\label{th:kerneldeg}
Let $k^{\bullet}_{n,k}$ be the number of rooted 3-graphs with
$n$ vertices and with root degree $k$. Let
$K^{\bullet}(x,w) = \sum k^{\bullet}_{n,k} x^n w^k$ be its associated
generating function. The following equation holds
\begin{equation}\label{eq:kerneldeg}
\begin{split}
K^\bullet (x,w) = B_0(x,w)\cdot
 C^\bullet\left(B_1(x),B_2(x),B_3(x,w)
\right)+A(x,w)
\end{split}
\end{equation}
where
$$
B_0(x,w) = e ^{(w^2-1)x^2/(2+2x) +  x(1-w)/(1+x)}, \qquad
B_1(x) = x e^{(x^2-2x)/(2+2x)},
$$
$$
B_2(x) = 2e^{-x/(1+x)}-1,\qquad
B_3(x,w) = \frac{(1+w)e^{-wx/(1+x)}-1}{2e^{-x/(1+x)}-1},
$$
$$
A(x,w) = A_0(x)+A_1(x)w+A_2(x)w^2,
$$
and $A_0(x)$, $A_1(x)$, $A_2(x)$ are analytic functions.
\end{theorem}
In order to prove this theorem we need some technical lemmas
that relate different classes of graphs.
\begin{lemma}
Let $\widetilde{C}^{\bullet}(x,w,z,y_1,\ldots,y_k,\ldots)$ be the generating
function of rooted connected planar weighted multigraphs where $x$ marks
vertices, $w$ marks the root degree, $z$ marks loops, and $y_k$ marks
$k$-edges. The following equation holds
\begin{equation}\label{eq:multigraphsdeg}
\widetilde{C}^\bullet (x,w,z,y_1,\ldots,y_k\ldots) =
e ^{z\cdot(w^2-1)/2} C^\bullet\left(x e^{z/2},\sum_{i\geq 1}\frac{y_i}{i!},
\frac{\sum_{i\geq 1}w^i\cdot y_i /i!}{\sum_{i\geq 1} y_i /i!}\right).
\end{equation}
\end{lemma}
\begin{proof}
Given a simple connected planar graph $\mathcal{G}$,
a connected planar multigraph
can be obtained from $\mathcal{G}$ by replacing each edge with a multiple edge,
and placing 0 or more loops in each vertex (see proof of
Corollary~\ref{cor:kernel} for details). In the case of rooted graphs,
if we replace an edge incident to the root with a $i$ edge, its root degree
is increased in $i-1$. Therefore, instead of replacing such an edge
with a multiple edge with generating function $y_i/i!$, we replace it with
a multiple edge with generating function $w^iy_i/i!$.
Similarly, when we add a loop incident to the root vertex, the root degree
is increased by 2. Therefore, its associated generating function is not
$z$, but $zw^2$.
\end{proof}

\begin{lemma}
Let $\widetilde{H}^{\bullet}(x,w,z,y_1,\ldots,y_k,\ldots)$ be the generating
function of rooted planar weighted 2-multigraphs where $x$ marks
vertices, $w$ marks the root degree, $z$ marks loops, and $y_k$ marks
$k$-edges. The following equation holds
\begin{equation}\label{eq:multicoresdeg}
\begin{array}{ll}
\widetilde{H}^\bullet (x,w,z,y_1,\ldots,y_k\ldots) =
e^{y_1x(1-w)}\widetilde{C}^\bullet (xe^{-y_1 x},w,z,y_1,\ldots,y_k\ldots) \\
\qquad -w\cdot A(x,z,y_1,\ldots y_k,\ldots) - x,
\end{array}
\end{equation}
for a given function $A(x,z,y_1,\ldots y_k,\ldots)$ that does not
depend on $w$.
\end{lemma}

\begin{proof}
The decomposition of a planar connected weighted multigraph into
its core and the attached rooted trees implies the following equation:
$$
\widetilde{C}^{\bullet}(x,w,z,y_1,\ldots,y_k,\ldots) =
\widetilde{H}^{\bullet}(T(x,y_1),w,z,y_1,\ldots,y_k,\ldots)
{T(x,y_1,w)\over T(x,y_1)}+
$$
$$
+\widetilde{H}^\bullet(T(x,y_1),z,y_1,\ldots,y_k,\ldots)
{wT(x,y_1,w)\over 1-T(x,y_1)}+T(x,y_1,w),
$$
where $T(x,y) = T(xy)/y$ is the generating function of
rooted trees where $x$ marks vertices and $y$ marks edges,
and $T(x,y,w) = T(xy,w)/y$ is the generating function of rooted
trees where $x$ marks vertices, $y$ marks edges, and $w$ marks the
root degree. The justification of this relation is analogous to
the proof of Theorem~\ref{th:coresdeg}, as well as
the inverse.
\end{proof}

\begin{lemma}
Let $K^{\bullet}(x,w)$ be the generating function of rooted simple planar
3-graphs where $x$ marks vertices and $w$ marks the root degree.
The following equation holds
\begin{equation}\label{eq:kernelcoresdegree}
K^{\bullet}(x,w) =
\widetilde H^{\bullet}(x,w,-sx,1+s,2s+s^2,\ldots ,ks^{k-1}+s^k,\ldots)
+w^2A(x),
\end{equation}
for a given function $A(x)$, and where $s=-x/(1+x)$.
\end{lemma}

\begin{proof}
The starting point is Equation (\ref{eq:multicoreskernel}), corresponding to
the decomposition of a planar 2-multigraph into its kernel and paths of vertices.
If we root a vertex of a planar 2-multigraph there are two options:
either it belongs to the kernel or it belongs to an edge of the kernel.
In the former case, its degree corresponds to the degree of the corresponding
vertex in the kernel. In the latter case its degree must be 2. With this
observation we can extend this equation so that it considers rooted graphs
and it takes the root degree into account, as
$$
\begin{array}{ll}
\widetilde{H}^\bullet(x,w,z,y_1,y_2,\ldots,y_k,\ldots)=&\\
\widetilde{K}^\bullet\left(x,w,sxy_1+xy_2+z,y_1+s,y_2+2y_1s+s^2,\ldots,
\sum_{j=0}^{k}\binom{k}{j}y_js^{k-j},\ldots\right) \\
\qquad +w^2A(x,z,y_1,\ldots,y_k,\ldots),
\end{array}
$$
where $A(x,z,y_1,\ldots,y_k,\ldots)$ does not depend
on $w$. This relation can be inverted as in Section~\ref{sec:equations},
and finally we can conclude (\ref{eq:kernelcoresdegree}) from
the following equation
$$
K^\bullet(x,w) = \widetilde K^\bullet(x,w,0,1,0,\ldots,0,\ldots).
$$
\end{proof}

\noindent {\it Proof of Theorem~\ref{th:kerneldeg}.}
Equation (\ref{eq:kerneldeg}) is a direct consequence of equations
(\ref{eq:kernelcoresdegree}), (\ref{eq:multicoresdeg}) and
(\ref{eq:multigraphsdeg}).

\begin{theorem}
Let $f_k$ be the limit probability that a random vertex has degree
$k$ in a planar 3-graph. The limit probability distribution
$p_K(w) = \sum f_k w^k$ exists and is computable.
\end{theorem}

\begin{proof}
The generating function $C^\bullet(x,y,w)$ is expressed
in~\cite{degree} as
$$
C^\bullet(x,y,w) = C_0(y,w)+C_2(y,w)X^2+C_3(y,w)X^3+O(X^4),
$$
where $X=\sqrt{1-x/\rho(y)}$. If we compose this expression with
(\ref{eq:kerneldeg}) we obtain
\begin{equation}\label{eq:kernelexp}
\begin{array}{ll}
K^\bullet(x,w) = B_0(x,w)\times\\ 
\left[C_0(B_2(x)),B_3(x,w))+
C_2(B_2(x),B_3(x,w))X^2+ 
C_3(B_2(x),B_3(x,w))X^3+O(X^4)\right]\\
\qquad +A(x,w),
\end{array}
\end{equation}
where $X=\sqrt{1-B_1(x)/\rho(B_2(x)}$. If we define
$D(x) = B_1(x)/\rho(B_2(x))$ then we can proceed as in the proof
of Theorem~\ref{th:3graphs}, obtaining that
$X=\sqrt{\tau D'(\tau)}Z+O(Z^2)$, where $Z = \sqrt{1-x/\tau}$.
Plugging this expression into~(\ref{eq:kernelexp}) we obtain
$$
K^\bullet(z,w) = K_0(w)+K_2(w)Z^2+K_3(w)Z^3+O(Z^4),
$$
where $Z = \sqrt{1-z/\tau}$ and
$$
K_3(w) = B_0(\tau,w)C_3(B_2(\tau),B_3(\tau,w))
(\tau D'(\tau))^{3/2}+a_0+a_1w+a_2w^2,
$$
for some constants $a_0$, $a_1$ and $a_2$. The limit probability distribution
of the root vertex being of degree $k$ is computed as
$$
p_K(w) = \frac{K_3(w)}{K_3(1)} =
\frac{B_0(\tau,w)C_3(B_2(\tau),B_3(\tau,w))(\tau D'(\tau))^{3/2}
+a_0+a_1w+a_2w^2}
{B_0(\tau,1)C_3(B_2(\tau),1)(\tau D'(\tau))^{3/2}+a_0+a_1+a_2}.
$$
Since we know that a 3-graph has no vertices of degree 0, 1 or 2,
we can choose suitable values of $a_0$, $a_1$ and $a_2$ such that
the probability distribution $p_K(w)=\sum f_kw^k$ satisfies
$f_0=f_1=f_2=0$. The function $C_3(y,w)$ is described in~\cite{degree},
and every other function that appears in the previous expression
is explicit. Therefore,
$p_K$ is computable, as we wanted to prove.
\end{proof}

We remark that $p_K(w)$ is expressed in terms of $C_3(x,w)$, which 
is a very involved (although elementary) function, given in the appendix in \cite{degree}.

\section{Concluding remarks}\label{sec:conclude}

Most of the results we have obtained can be extended to other classes of graphs. Let $\G$ be a class of graphs closed under taking minors such that the excluded minors of $\G$ are 2-connected. Interesting examples are the classes of series-parallel and outerplanar graphs.
Given such a class $\G$, a connected graph is in $\G$ if and only if its core is in~$\G$. Hence  Equation (\ref{eq:cores}) also
holds for graphs in $\G$.  Using the results from \cite{SP}, we have performed the corresponding computations for the classes of series-parallel and outerplanar graphs (there are no results for kernels since outerplanar and series-parallel have always minimum degree at most two). The results are displayed in the next table, together with the data for planar graphs.
The expected number of edges is $\mu  n$, and the expected size of the core is $\kappa  n$. It is worth remarking that the size of the core is always linear, whereas the size  of the largest block in series-parallel and outerplanar graphs is only $O(\log n)$ \cite{3-conn,KS}.

$$
\renewcommand{\arraystretch}{1.5}
\begin{array}{|l|c|c|c|}
\hline
\hbox{Graphs} & \hbox{Growth constant} & \mu \hbox{ (edges) } & \kappa \hbox{ (core)} \\
\hline
\hbox{Outerplanar} & 7.32  & 1.56 & 0.84\\
\hline
\hbox{Outerplanar 2-graphs} & 6.24  & 1.67 & \\
\hline
\hbox{Series-parallel} & 9.07  & 1.62 & 0.875 \\
\hline
\hbox{Series-parallel 2-graphs} & 8.01  & 1.70 & \\
\hline
\hbox{Planar} & 27.23  & 2.21 & 0.962 \\
\hline
\hbox{Planar 2-graphs} & 26.21  & 2.26 & \\
\hline
\end{array}
$$

\bigskip
The $k$-core of a graph $G$ is the maximum subgraph
of G in which all vertices have degree at least $k$. Equivalently, it is the subgraph of G formed by deleting repeatedly (in any order) all vertices of degree less than $k$. In this terminology, what we have called the core of a graph is the 2-core. Since a random planar graph contains linearly many copies of any fixed connected planar graph  \cite{colin2009,gn} it is not difficult to show that the 3-core, 4-core and 5-core of a random planar graph have all linear size with high probability (there is no 6-core since a planar graph has always a vertex of degree at most five). The interesting question is however whether the $k$-core has a connected component of linear size, as is the case for $k=2$.
We have performed computational experiments on random planar graphs, using the algorithm described in \cite{fusy}, and based on the results we formulate the following conjecture.

\paragraph{Conjecture.} With high probability the 3-core of a random planar graph has one component of linear size.
With high probability the components of the 4-core of a random planar graph are all of  sublinear size.

\bigskip

We have not been able to prove neither of the conjectures. As opposed to the kernel, the 3-core is obtained by repeatedly \emph{removing} vertices of degree two. These deletions may have long-range effects that appear difficult to analyze. Even more challenging appears the analysis of the 4-core.

\paragraph{Acknowledgements.}

Part of this work was done while the second author was visiting the Technical University of Vienna. 
We are  very grateful to Michael Drmota for his help on several technical points of the proofs of our results.

\section*{Appendix}

The following technical result was needed to conclude the proof of Theorem \ref{th:3graphs}.

\begin{lemma*}\label{le:K.Delta}
	The generating function $K(x)$ is  
	$\Delta$-analytic at its dominant singularity~$\tau$. 
\end{lemma*}
\begin{proof}
	For the proof we  introduce the following generating functions:
	\begin{itemize}
		\item $K^\bullet(x)$ is the generating function
		of rooted planar graphs with minimum degree
		at least 3. Note that $K^\bullet(x)$
		has the same radius of convergence $\tau$ as $K(x)$.
		\item For $i=1,2$, $K^\bullet_i(x)$ is the generating
		function of rooted planar graphs where all the
		vertices have degree at least 3 except
		for the root, which has degree exactly $i$.
		
		\item $\widehat B(x,u)$ is the generating function
		of 2-connected planar graphs
		where $x$ marks vertices of degree at least 3,
		$u$ marks vertices of degree exactly two,
		and both types of vertices are labelled with
		the same set of labels. In particular
		$\widehat B(x,u) = \sum_{n,m\geq 0}b_{n,m}
		{x^nu^m/ (n+m)!}$, where $b_{n,m}$
		counts  2-connected planar graphs
		with $n$ vertices of degree at least 3 and
		$m$ vertices of degree exactly 2. Note that
		we do not count a single edge in $\widehat B(x,u)$
		since it has no vertices of degree 2 or more.
	\end{itemize}
	A simple combinatorial argument gives 
	\begin{equation}\label{eq:K}
	\begin{array}{ll}
	K^\bullet_1 &= F_1(x,K^\bullet,K^\bullet_2), \\
	K^\bullet_2 &= F_2(x,K^\bullet,K^\bullet_1,K^\bullet_2), \\
	K^\bullet &= F_3(x,K^\bullet,K^\bullet_1,K^\bullet_2),
	\end{array}
	\end{equation}
	where
	$$
	\begin{array}{ll}
	F_1(x,z,z_2)  = x (z+z_2), \\
	F_2(x,z,z_1,z_2) = x\left(\ds{(z+z_2)^2
		\over 2}+B_u\right), \\ 
	F_3(x,z,z_1,z_2) =  		\\
\qquad	x\left( B_x+(z+z_2)
	(B_u+B_x)+\ds{(B_u+B_x)^2\over 2}+
	\exp_{\geq 3}(z+z_2+B_u+B_x)\right),
	\end{array}
	$$ 
	and 
	$$
	\begin{array}{ll}
	B_x = \widehat B_x(x+z+z_1+z_2,
	z+z_1+z_2),\\
	B_u = \widehat B_u(x+z+z_1+z_2,
	z+z_1+z_2).
	\end{array}
	$$
	We remark  that the coefficients of the series $F_i$ are non-negative.
	
	First we   check that $F_1$, $F_2$
	and $F_3$ are analytic in a neighbourhood
	of~$0$, which  is equivalent to checking
	that $B_x$ and $B_u$ are analytical at $0$.
	We derive this from the following properties of $\widehat B$:
	\begin{itemize}
		\item $\widehat B$ is a  series in  $x$ and $u$ with non-negative coefficients.
		\item $B_x$ and $B_u$ are analytic at $(x_0,u_0)$
		if and only if $\widehat B$ is analytic at
		$(x_0,u_0)$.
		\item If $\widehat B$ is analytic
		near $(x_0,u_0)$, then it is analytic
		at $(x_1,u_1)$, for $|x_1|\leq x_0$
		and $|u_1|\leq u_0$.
		\item $\widehat B (x,x)=B(x)-x^2/2$, hence $\widehat B(x,u)$ is for $(x_0,u_0) < (R,R)$,
		although it might be analytic for $(x_0,u_0)$
		where $u_0 < R \leq x_0$ or $x_0 < R \leq u_0$ as well.
	\end{itemize}
	This implies that $\widehat{B}(x,u)$ is analytic
	at $0$, and the same holds for $F_1$, $F_2$ and $F_3$.
	Since 
	$K^\bullet(0) = K^\bullet_1(0) = K^\bullet_2(0) = 0$
	we have that the system (\ref{eq:K}) holds in a neighbourhood of $x = 0$, and it ceases to hold at the singularity
	of $K^\bullet$. First note that $K^\bullet_1$,
	$K^\bullet_2$ and $K^\bullet$ have all the same radius
	of convergence, $\tau$, because all of them are the
	sum of the others plus some positive terms. 
	In these cases there are three sources
	of singularities:
	\begin{itemize}
		\item Poles at $F_1$, $F_2$ and $F_3$. This is not
		possible, since all the involved functions are 
		analytic in $\mathbb{C}$ except for $B_x$ and $B_u$.
		\item Branching point in solving $F_1$, $F_2$ and
		$F_3$. This is not possible either, since in this case
		the singular analysis of $K^\bullet$ would be of the
		form $K^\bullet = K^\bullet_0 + K^\bullet_1 Z+O(Z^2)$,
		where $K^\bullet_1\neq 0$, and we have seen
		in Equation (\ref{eq:Ksing}) that this is not the case.
		\item A singularity in $B_x(x)$ and $B_u(x)$ 
		(note that both functions share singularities). This must be
		the source of singularity, and in fact
		the singularity of $B_x(x)$ must be exactly at $x=\tau$.
		If the singularity was at a given $x_0<\tau$, then
		there would be an unbounded derivative of $B_x$ at $x_0$,
		and since $K^\bullet$ is $x\cdot B_x$ plus some
		positive terms, then $K^\bullet$ would have
		an unbounded derivative at $x_0<\tau$, and that is
		impossible since $K^\bullet$ is analytic for
		$x$ with $|x|<\tau$.
		The singularity cannot be at an $x_0>\tau$ either, 
		because we discarded the other sources of singularities
		and this would imply that $K^\bullet$ is analytic
		for $x>\tau$, which is impossible.
	\end{itemize}
	Therefore the equations hold for $x$ such that
	$B_x$ is analytic at
	$(x+K^\bullet_1(x)+K^\bullet_2(x)
	+K^\bullet(x),K^\bullet_1(x)+K^\bullet_2(x)
	+K^\bullet(x))$.
	Now, consider $x$ such that $|x|=\tau$ but 
	$x \neq \tau$. Then, by positivity of $\widehat B$
	and $K^\bullet_i$, we have:
	$(|x+K^\bullet_1(x)+K^\bullet_2(x)
	+K^\bullet(x)|,|K^\bullet_1(x)+K^\bullet_2(x)
	+K^\bullet(x)|)<(\tau+K^\bullet_1(\tau)+K^\bullet_2(\tau)
	+K^\bullet(\tau),K^\bullet_1(\tau)+K^\bullet_2(\tau)
	+K^\bullet(\tau))$, so $\widehat B$ is analytic
	and the equations hold.
	Therefore $K^\bullet(x)$
	is analytic as well.
	
	We just have to check that, if $|x|=\tau$ and 
	$x\neq \tau$ then there is not branching point
	when solving the system of equations.
	Let $A$ be the Jacobian matrix of $(F_1,F_2,F_3)$.
	According to \cite[Section 2.2.5]{drmotatrees},
	the maximum positive eigenvalue of $A$ is a positive
	function in $x$, $K^\bullet_i$. We know that 
	such an eigenvalue must be smaller than $1$ when evaluated
	at $(\tau,K_1^\bullet(\tau),K_2^\bullet(\tau),K^\bullet(\tau))$,
	since otherwise there would exist a real $x$ with
	$|x|<\tau$ such that the system evaluated at $x$ has a branching
	point, and we know this is not possible.
	Hence, by positivity of the maximum eigenvalue,
	if $|x|=\tau$ but $x\neq \tau$ then the maximum eigenvalue
	of $A$ evaluated at
	$(x,K_1^\bullet(x),K_2^\bullet(x),K^\bullet(x))$ cannot be 1,
	so we can apply the Implicit Function Theorem and there is
	an analytic continuation of $K^\bullet$ in a neighbourhood of $x$,
	and by compactness
	it can be extended to a $\Delta$-domain,
	as we wanted to prove. 
\end{proof}

\end{document}